\documentclass[reqno,11pt]{amsart}
\usepackage[utf8]{inputenc}

\usepackage[left=1.5in, right=1.5in, top=1in]{geometry}
\usepackage{graphicx}
\usepackage{amsmath,amssymb,mathrsfs,color,times,textcomp,verbatim}
\usepackage{amsthm}
\usepackage{xcolor}
\usepackage{enumitem}
\usepackage[colorlinks=true]{hyperref}
\hypersetup{urlcolor=blue, citecolor=red, linkcolor=blue}
\usepackage{amscd}

\usepackage[square,numbers]{natbib}

\numberwithin{equation}{section}
\theoremstyle{plain}
\newtheorem{theorem}{Theorem}[section]
\newtheorem{proposition}[theorem]{Proposition}
\newtheorem{lemma}[theorem]{Lemma}

\newtheorem{claim}{Claim}

\theoremstyle{definition}
\newtheorem{definition}[theorem]{Definition}

\newtheorem{remark}[theorem]{Remark}

\usepackage{graphicx}

\def \bnped{B_+^N(0,2\delta)}
\def \dze{D(0,\delta)}

\def \bnpe{B_+^N(0,\delta)}
\def \RpN{\mathbb{R}_+^N}
\def \Ygs{\mathcal{Y}^\gamma_{\mathbb{S}^n}}
\def \ve{\varepsilon}
\def \Rn{\mathbb{R}^n}
\def \crexp{\frac{2n}{n-2\gamma}}
\def \iRN{\int_{\RpN}}
\def \p{\partial}
\def \l{\lambda}
\newcommand{\N}{\mathbb{N}}

\title{Uniformization Theorems: Between Yamabe and Paneitz}

\author{Cheikh Birahim Ndiaye}
\address{Department of Mathematics of Howard University, 217 Annex 3 Graduate School 4th St NW \& College St NW, Washington, DC 20059k, USA.}
\email{cheikh.ndiaye@howard.edu}

\author{Yannick Sire}
\address{Department of Mathematics, Johns Hopkins University, 3400 N. Charles, st.,
 Baltimore, MD 21218, USA.}
 \email{sire@math.jhu.edu}
\author{Liming Sun}
\address{Department of Mathematics, Johns Hopkins University, 3400 N. Charles st., 
 Baltimore, MD 21218, USA.}
\email{lsun@math.jhu.edu}

\date{\today}

\begin{document}

\maketitle

\begin{abstract}
This paper is devoted to several existence results for a generalized version of the Yamabe problem. First, we prove the remaining global cases for the range of powers $\gamma\in (0,1)$ for the generalized Yamabe problem introduced by Gonzalez and Qing. Second, building on a new approach by Case and Chang for this problem, we prove that this Yamabe problem is solvable in the Poincar\'{e}-Einstein case for $\gamma\in (1,\min\{2,n/2\})$ provided the associated fractional GJMS operator satisfies the strong maximum principle.

\end{abstract}

\tableofcontents
\setlength{\parskip}{4pt}
\section{Introduction}

The resolution of the Yamabe conjecture, i.e.\,the problem of finding a constant scalar curvature metric in a given conformal class on closed manifolds, has been a landmark in Geometric Analysis after the works of \citep{yamabe1960deformation,trudinger1968remarks,aubinsome,schoen1984conformal}. Several generalizations to different ambient manifolds appeared after this series of works (e.g. \cite{gamara2001cr,gamara2001lcf,ahmedou,GonzalezNogueras2013}).

\vspace{5pt}
We consider here some rather recent development whose foundation can be found in a seminal paper by Graham and Zworksi \cite{Graham2003} about a new and fruitful approach to the realization of the GJMS operators. Suppose that $(X^{n+1},g_+)$ is a Poincar\'e-Einstein (P-E) manifold with dimension $n\geq 2$, that is, a conformally compact Riemannian manifold with $Ric(g_+)=-ng_+$. Assume that $(X^{n+1},g_+)$ has a conformal infinity $(M^n,[h])$, where $M$ is a compact manifold.
There is a family of conformally covariant operator $P_{{ h }}^\gamma$ ($\gamma\in (0,\frac{n}{2})$) discovered by \cite{Graham2003} which satisfies
\[P^\gamma_{{ h }}(uf)=u^{\frac{n+2\gamma}{n-2\gamma}}P^\gamma_{\tilde h}(f)\]
where $\tilde h=u^\frac{4}{n-2\gamma}{ h }$. Then one can define the so called $Q^\gamma$-curvature as $Q^\gamma_{{ h }}=P^{\gamma}_{{ h }}(1)$. These operators $P_h^\gamma$ appear to be the
higher-order generalizations (for $\gamma >1$) of the conformal Laplacian (including the Paneitz operator for $\gamma=2$). They coincide with
the GJMS operators \cite{graham1992conformally} for suitable integer values of $\gamma$. Specially, $Q_h^\gamma$ is just the scalar curvature for $\gamma=1$, and the $Q$-curvature for $\gamma=2$. This new notion of curvature has been investigated in \citep{qing2006compactness,chang2011fractional,del2012singular,GonzalezNogueras2013,Kim2018}. When $\gamma=\frac12$, $Q_h^\gamma$ is just the mean curvature of $(M,h)$ in $(X,g)$.

Keeping in mind the purpose of the Yamabe conjecture, one aims at finding a conformal metric $h\in [{ h }]$ such that $Q^\gamma_{h}$ is constant. Since the parameter $\gamma$ ranges from $0$ to $\frac{n}{2}$, this provides a $1$-parameter family of metrics and sheds some new light on classical constant curvature prescription problems. Following \cite{GonzalezNogueras2013}, solving the problem is equivalent to find a critical point of the following Euler-Lagrange functional
\begin{align}
    \mathcal{E} _ { { h } } ^ { \gamma } [ u ] = \frac { \oint _ { M } u P _ { { h } } ^ { \gamma } u d \mu_ { { h } } } { \left( \oint _ { M }  u  ^ { \frac { 2 n } { n - 2 \gamma } } d\sigma_ { { h } } \right) ^ { \frac { n - 2 \gamma } { n } } } \quad \text { for } u \in W _+^ { \gamma,2 } ( M,h ) \backslash \{ 0 \},
\end{align}
where $W^{\gamma,2}(M,h)$ denotes the usual fractional Sobolev space on $M$ with respect to Riemannian metric $h$, and $W^{\gamma,2}_+(M,h)=W^{\gamma,2}(M,h)\cap \{u\geq 0\}$. The infimum is called the $\gamma$-Yamabe constant
\[\mathcal{Y}^\gamma(M,[{ h }])=\inf_{W ^ { \gamma,2 } ( M,h ) \backslash \{ 0 \}}\mathcal{E} _ { { h } } ^ { \gamma } [ u ].\]
The critical points of $\mathcal{E} _ { { h } } ^ { \gamma }$  satisfy
\begin{align}\label{eq:goal}
    P _ { { h } } ^ { \gamma } u=cu^{\frac{n+2\gamma}{n-2\gamma}}, \quad u\geq 0
\end{align}
for some constant $c$. If $P _ { { h } } ^ { \gamma }$ satifies the strong maximum principle, or its Green's function is positive, then $u$ is strictly positive and  satisfy the above equality. Hence, $u^{\frac{4}{n-2\gamma}}h$ is a conformal metric whose fractional curvature is constant. \citet{GonzalezNogueras2013} prove that $P_{{ h }}^\gamma$ has a strong maximum principle when $\gamma\in (0,1)$. For higher $\gamma$, in the setting of Poincar\'{e}-Einstein $(X^{n+1},g_+)$ with conformal infinity $(M^n,[{ h }])$, \citet{CaC} proved that if $(M,[{ h }])$ has scalar curvature $R_h\geq 0$ and $Q_h^\gamma\geq 0$ and $Q_h^\gamma\not\equiv 0$ for $1<\gamma<\min\{2,n/2\}$, then $P_{{ h }}^\gamma$ has a strong maximum principle.

The present paper is two-fold. First, we complete the work \cite{GonzalezNogueras2013,Gonzalez2018, Kim2018,Mayer2017} providing existence results in some range of dimensions depending on $\gamma \in (0,1)$. Our arguments also apply to the general asymptotically hyperbolic (AH) manifolds. Second, for the higher order $1<\gamma< \min\{2,n/2\}$, when $X$ is a Poincar\'e-Einstein manifold, we completely solve the fractional Yamabe problem under the assumption of the strong maximum principle.  

In the present contribution, we consider two types of situations, denoted below \textbf{Type I} and \textbf{Type II}.

First, we consider \textbf{Type I}; that is $\gamma\in(0,1)$. Assume that $(X^{n+1}, g_+)$ is a P-E manifold with conformal infinity $(M,[h])$. \citet{Kim2018} and \citet{kim2017conformal} showed that if $n\geq 4+2\gamma$ and $M$ is non-locally conformally flat then $\gamma$-Yamabe problem is solvable. \citet{Mayer2017} proved the solvability for $M$ being locally conformally flat. Hence, the remaining case of Type I in P-E setting is the low dimensional case
\begin{enumerate}[labelsep=0ex,align=left]
    \item[Case (I-1)]: $(X^{n+1}, g_+)$ is P-E with $(M,[h])$ and  $n< 2\gamma+4$.
\end{enumerate}
that is, $n=3,4$ when $\gamma\in(0,1)$ and $n=5$ when $\gamma\in (0,\frac12)$. If $(X^{n+1},g_+)$ is just AH, the second fundamental form of $(M,h)$ will come into play. One needs consider whether $(M,h)$ is umbilic or not, which induce many different cases. Readers are directed to \citet{Kim2018} with additional assumption. Nevertheless, our method also apply to the  lower dimensional case in AH setting
\begin{enumerate}[labelsep=0ex,align=left]
\item[Case (I-2)]: $(X^{n+1}, g_+)$ is AH with $(M,[h])$ and $n<2+2\gamma$.
\end{enumerate}

Second, we consider \textbf{Type II}; that is $\gamma\in(1,\min\{2,n/2\})$.
The main contribution of the present paper is to deal with the higher order fractional Yamabe type problems. Assume that $(X^{n+1},g_+)$ is a P-E manifold with conformal infinity $(M,[h])$ for the following cases
\begin{enumerate}[labelsep=0ex,align=left]
    \item[Case (II-1)]: Low dimension, $n<2\gamma+4$,
    \item[Case (II-2)]: $(M,[h])$ is locally conformally flat,
    \item[Case (II-3)]: $n>2\gamma+4$  and $(M,[h])$ is non-locally conformally flat,
    \item[Case (II-4)]: $n=2\gamma+4$ and $(M,[h])$ is non-locally conformally flat.
\end{enumerate}

To attack the above cases, we need to notice the distinctive nature of them. (I-1), (I-2), (II-1), (II-2) are ``Global" cases and (II-3) and (II-4) are ``Local" cases.  Let us recall what is commonly called Local and Global cases in the Geometric Analysis community. 
Take the classical Yamabe problem for example, that is, $\gamma=1$.
With this agreement in mind and recalling that the functional is  $\mathcal{E}_h^1$  and the the standard bubble is  $U_{a, \varepsilon}$ (see \eqref{U:I-1} and \eqref{U:II-3}, we omit  $\delta$  for simplicity)), then by a standard Taylor expansion, and using the explicit form (decay) of  $U_{a, \varepsilon}$, one has the following formula
\begin{equation}\label{defloc-glo}
\mathcal{E}_h^1(U_{a, \varepsilon})=\mathcal{Y}^1_{\mathbb{S}^n}-\sum_{i=1}^{n-3}\mathcal{L}_i(a)\varepsilon^i-\mathcal{L}_{n-2}(a)\varepsilon^{n-2}\ln \varepsilon- \mathcal{M}_{n-2}(a)\varepsilon^{n-2}+o(\varepsilon^{n-2}).
\end{equation}
The case is called  {\bf Local}  if  $\exists \; a\in M,  i\in \{1,\cdots, n-2\}:  \mathcal{L}_i(a)\neq 0$ and it is referred to {\bf Global}  if    $
\forall\; a\in M$,  $  \forall\, i\in\{1, \cdots, n-2\}:   \mathcal{L}_i(a)=0$. The coefficient $\mathcal{M}_{n-2}$ is associate to the ``mass" at $a$. The Global case means the terms higher than mass should all vanish. When $\gamma\neq 1$, the mass term should have order $\ve^{n-2\gamma}$ in \eqref{defloc-glo}. Roughly speaking, in P-E setting, since {the first term in the above expansion is $\ve^4$ with coefficient the norm of the Weyl tensor (up to a non-zero factor), and that when the Weyl tensor is identically zero automatically all the coefficients in the above expansion until the logarithmic term vanish, then one can see how the property of being locally conformally flat and the competition between $\ve^{n-2\gamma}$ and $\ve^4$ describe fully the Local and Global cases. However, in AH setting, on top of the latter considerations one has additional terms starting at $\ve^2$ with coefficient the norm of the trace free part of the second fundamental form of $(M,h)\subset (X,g)$ up to a non-zero factor. If $M$ is umbilical, then the expansion is the same as in the case of P-E. Hence, in AH, the umbilicity, the locally conformally flatness, the size $\ve^2,\ve^4$, and $\ve^{n-2\gamma}$ describe the Global and Local cases.
}

To solve the Local cases, it is enough, in most of the arguments, to use the local $U_{a,\ve}$ (see \eqref{U:II-3}). For the Global cases, besides the work of Schoen \cite{schoen1984conformal}, there is also an indirect method through Algebraic Topological arguments by \citet{Bahri1988} (also called Barycenter Technique). Later \citet{bahri1989critical} developed the theory of critical points at infinity for Yamabe problems on Euclidean domains. We refer the reader to its applications in the locally conformally flat case in \cite{ bahri1993another} and in the low dimensional case in \cite{bahri1996non}. Adapting the Barycenter Technique to the fractional Yamabe problem, we achieve the following theorem

\begin{theorem}\label{gamma01}
Suppose that $\gamma\in (0,1)$ and $n\geq 2$. Assume $(X^{n+1},g_+)$ is a Poincar\'{e}-Einstein manifold with  conformal infinity $(M^n,[{ h }])$ with $\lambda_1(-\Delta_{g_+})>\frac14n^2-\gamma^2$. If $\mathcal{Y}^\gamma(M,[h])>0$, then there exists $h\in [{ h }]$ such that $Q_{h}^\gamma$ is a constant. 
\end{theorem}

The previous theorem solves completely the Yamabe problem for the $Q^\gamma$-curvature, complementing the works \cite{GonzalezNogueras2013,Gonzalez2018,Kim2018,Mayer2017} in the Poincar\'e-Einstein setting. We will provide an additional result on the more general framework of AH manifolds in the last section. 
\begin{theorem}\label{gamma12}
Suppose that $1<\gamma< \min\{2,n/2\}$ and $n\geq 3$. Assume $(X^{n+1},g_+)$ is a Poincar\'{e}-Einstein manifold with conformal infinity $(M^n,[{ h }])$ with $\lambda_1(-\Delta_{g_+})>\frac14n^2-(2-\gamma)^2$. If $R_h\geq 0$ and $Q_h^\gamma\geq 0$ and $Q_h^\gamma\not\equiv 0$ for some $h\in[h]$, then there exists some $\tilde h\in [{ h }]$ such that $Q_{\tilde h}^\gamma$ is a constant. 
\end{theorem}

To prove our results in the Local cases we employ  \citet{aubinsome}-\citet{schoen1984conformal}'s Minimizing Technique. In the global cases, we use the Algebraic Topological argument of Bahri-Coron \citep{Bahri1988}. Since most of this work is concerned with Global cases, and moreover to find excellent exposition of the Aubin-Schoen's Minimizing Technique seems not to be difficult (see for example \citet{lee1987yamabe}), then we decide to discuss just how the Barycenter Technique of Bahri-Coron works in finding a critical point. We just point out that in our application of \citet{aubinsome}-\citet{schoen1984conformal}'s Minimizing Technique, we took a short-cut by bringing into play the Eckeland Variational Principle. {We chose this approach not only to shorten the exposition, but to also emphasize the common point between the Aubin-Schoen minimizing technique and Algebraic Topological argument to Bahri-Coron.}

The Algebraic Topological argument of Bahri-Coron \citep{Bahri1988} is based on two fundamental facts: the quantization of $(\mathcal{E}_h^\gamma)^{\frac{n}{2\gamma}}$ (see Lemma \ref{lem:profile-decomp}) and the strong interaction phenomenon (see Lemma \ref{lem:interaction}). Readers can find a detailed explanation of Barycenter Technique in \citet{mayer2017barycenter}. Here we just sketch the main idea behind it.


On one hand, the argument needs a starting point, which is the existence of a topological class\ $X_1$  which is non-zero in the   $\mathbb{Z}_2$-homology of some lower sub-level set   $L_c:=\{u
:  (\mathcal{E}_h^\gamma[u])^{\frac{n}{2\gamma}}\leq c\}$. Here one starts with  $c=(\mathcal{Y}^\gamma_{\mathbb{S}^n})^{\frac{n}{2\gamma}}+\varepsilon_1$  for some $\varepsilon_1>0$, and the existence of  $X_1$  is ensured by  $H_n(M, \mathbb{Z}_2)\neq 0$  and bubbling (See Lemma \ref{eq:nontrivialf1}).

Then, the next step is to start \textit{piling up masses}     $v_{a,\varepsilon,\delta} $ (see its definition \eqref{def:V_i}) over   $X_1$, thereby moving from the level    $(\mathcal{Y}^\gamma_{\mathbb{S}^n})^{\frac{n}{2\gamma}}+\varepsilon_1$    to the level     $2(\mathcal{Y}^\gamma_{\mathbb{S}^n})^{\frac{n}{2\gamma}}+\varepsilon_1$, from 
the level   $2(\mathcal{Y}^\gamma_{\mathbb{S}^n})^{\frac{n}{2\gamma}}+\varepsilon_1$    to  the level    $3(\mathcal{Y}^\gamma_{\mathbb{S}^n})^{\frac{n}{2\gamma}}+\varepsilon_1$,   $ \cdots$,    from the level     $p(\mathcal{Y}^\gamma_{\mathbb{S}^n})^{\frac{n}{2\gamma}}+\varepsilon_1$    to the level    $(p+1)(\mathcal{Y}^\gamma_{\mathbb{S}^n})^{\frac{n}{2\gamma}}+\varepsilon_1$,     so on.
At each step, as one moves from  the level   $p(\mathcal{Y}^\gamma_{\mathbb{S}^n})^{\frac{n}{2\gamma}}+\varepsilon_1$   to the level    $(p+1)(\mathcal{Y}^\gamma_{\mathbb{S}^n})^{\frac{n}{2\gamma}}+\varepsilon_1$, one constructs a non-zero topological class    $X_{p+1}$    which reads    $(1-t)u + tv_{a,\varepsilon},     u\in X_p,     t\in  [0,  1]$ (see Lemma \ref{lem:nontrivialrecursive}).

However, because of the strong interaction phenomenon, for  $p_0$  large, we are passing from the level    $p_0(\mathcal{Y}^\gamma_{\mathbb{S}^n})^{\frac{n}{2\gamma}}+\varepsilon_1$   to the level    $(p_0+1)(\mathcal{Y}^\gamma_{\mathbb{S}^n})^{\frac{n}{2\gamma}}-\bar\varepsilon_1$   for some     $\bar\varepsilon_1>0$.
Then, assuming that there is no solution, we reach a contradiction since, as a result of the quantization phenomenon,  $(PS)_c$  holds    $\forall \,c$    such that    $p_0(\mathcal{Y}^\gamma_{\mathbb{S}^n})^{\frac{n}{2\gamma}}+\ve_1\leq  c\leq (p_0+1)(\mathcal{Y}^\gamma_{\mathbb{S}^n})^{\frac{n}{2\gamma}}-\bar\varepsilon_1$.


{We were assuming $R_h\geq 0$ and $Q_h^\gamma\geq 0$ and $Q_h^\gamma\not\equiv 0$ in the Theorem \ref{gamma12}, because we need that $P_{{ h }}^\gamma$ satisfies the strong maximum principle, which is proved by \citet{CaC} under these assumptions.} We conjecture that our results hold for all $\gamma\in (0,\frac{n}{2})$ provided $P_{{ h }}^\gamma$ satisfies the strong maximum principle.

This article is organized as follows. In Section 2, we recall some basic notions of smooth metric measure space and the fractional GJMS operators, which are contained in \cite{Case2017}. We define the standard bubbles for $\gamma\in (0,1)$ and $\gamma\in(1,\min\{2,n/2\})$ respectively and list their properties need for the remaining sections. In section 3 and 4, we define some test function $U_{a,\ve,\delta}$ and calculate their energy $\overline{\mathcal{E}}_h^\gamma[U_{a,\ve,\delta}]$ for different cases respectively. In Section 5, we stated the profile decomposition for the Palais-Smale sequences of $\mathcal{E}_h^\gamma$ and proved all the Local cases. The crucial interaction estimate between bubbles are also established in this section. In Section 6, the algebraic topological argument is applied to all Global cases. Section 7 illustrate the adaption to asymptotically hyperbolic case. Some necessary estimates are established in the Appendix at the end.

\section{Preliminaries}
In this section, we shall first describe the notions of smooth metric measure spaces and the fractional GJMS operators. After that we will define the standard bubbles and state their properties.
\subsection{Smooth metric measure spaces and fractional GJMS operators}\hspace*{\fill} 

A triple $(X^{n+1},M^n, g_+)$ is a \textit{Poincar\'e-Einstein} manifold if
\begin{enumerate}
\item[(1)] $X^{n+1}$ is (diffeomorphic to) the interior of a compact manifold $\bar X^{n+1}$ with boundary $\partial X=M^n$,
\item[(2)] $(X^{n+1},g_+)$ is complete with $Ric(g_+)=-ng_+$, 
\item[(3)] there exists a nonnegative $\rho\in C^\infty (X)$ such that $\rho^{-1}(0)=M^n$, $d\rho\neq 0$ along $M$, and the metric $g:=\rho^2g_+$ extends to a smooth metric on $\bar X^{n+1}$.
\end{enumerate}
A function $\rho$ satisfying these properties is called \textit{defining function}. Since $\rho$ is only determined up to multiplication by a positive smooth function on $\bar X$, it is clear that only the conformal class $[h]:=[g|_{TM}]$ on $M$ is
well-defined for a Poincar\'e-Einstein manifold. We call the pair $(M^n,[h])$ the conformal boundary of the Poincar\'e-Einstein manifold $(X^{n+1},M^n, g_+)$, and we call a metric $h\in [h]$ a representative of the conformal boundary. To each such representative there is a defining function $\rho$, unique in a neighborhood of $M$ and called the  \textit{geodesic defining function}. Moreover, $g_+$ has normal form $g_+=\rho^{-2}(d\rho^2+h_\rho)$ near $M$, where $h_\rho$ is a one-parameter family of metrics on $M$ satisfying $h_0=h$. $h_\rho$ has an asymptotic expansion which contains only even powers of $\rho$, at least up to degree $n$. For a more intrinsic discussion of these topics, we refer the reader to \cite{Graham2003}.

A smooth metric measure space (SMMS) is a four-tuple $(\bar X^{n+1},g,\rho,m)$ formed from a smooth manifold $\bar X^{n+1}$ with (possibly empty) boundary $M^n=\partial \bar X$, a Riemannian metric $g$ on $\bar X$, a nonnegative function $\rho\in C^\infty(\bar X)$ with $\rho^{-1}(0)=M$, and a dimensional constant $m\in (1-n,\infty)$. Formally, the interior of $\bar X$, denoted as $X$,  represents the base of a warped product
\begin{align}\label{eq:warp-prod}
(X^{n+1}\times \mathbb{S}^m, g\oplus \rho^2d\theta^2)
\end{align}
where $(\mathbb{S}^m,d\theta^2)$ the $m$-sphere with the metric of constant sectional curvature one. The geometric invariants defined on a SMMS are obtained by considering their Riemannian counterparts on \eqref{eq:warp-prod} while restricting to the base $X$, and then extend the definition to general $m\in (1-n,\infty)$ by treating $m$ as a formal variable. The weighted Laplacian $\Delta^m_{\rho}:C^\infty(X)\to C^ \infty(X)$ is defined as
\[\Delta^m_\rho U:=\Delta_gU+m\rho^{-1}\langle\nabla \rho, \nabla U\rangle_g,\quad U\in C^\infty(X)\]
which is a formally self-adjoint operator with respect to the measure $\rho^md\mu_g$. Here $d\mu_g$ is the volume element of $g$. The weighted Schouten scalar $J^m_\rho$ and weighted Schouten tensor $P^m_\rho$ are
\begin{align*}
J_{\rho}^{m} :=&\frac{1}{2(m+n)}\left(R-2 m \rho^{-1} \Delta \rho-m(m-1) \rho^{-2}\left(|\nabla \rho|^{2}-1\right)\right)\\
P_{\rho}^{m} :=&\frac{1}{m+n-1}\left(\operatorname{Ric}-m \rho^{-1} \nabla^{2} \rho-J_{\rho}^{m}\right)
\end{align*}

We shall confine ourself to a special type of SMMS,
\begin{definition}
A \textbf{geodesic} SMMS $(\bar X,g:=\rho^2g_+,\rho,m)$ is generated by a Poincar\'e-Einstein manifold $(X^{n+1}, M^n,g_+)$ and a geodesic defining function $\rho$ near $M$, that is $|\nabla \rho|_g=1$ near $M$. 
\end{definition}
For a geodesic SMMS, the weighted Schouten scalar and tensor take simpler forms. By \citet[Lemma 3.2]{CaC}, we have $J^m_\rho=J$ the Schouten scalar of $(\bar X,g)$ and $P^m_\rho=P$ the Schouten tensor of $(\bar X,g)$. On a geodesic SMMS, the weighted conformal Laplacian $L_{2,\rho}^m$ and weighted Paneitz operator $L_{4,\rho}^m$ are defined as 
\begin{align*} 
L_{2, \rho}^{m} U & :=-\Delta_{\rho}^m U+\frac{1}{2}(m+n-1) J \cdot U \\ 
L_{4, \rho}^{m} U & :=\left(-\Delta_{\rho}^m\right)^{2} U+\delta_{\rho}\left(\left(4 P-(m+n-1) J g\right)(\nabla U)\right)+\frac{1}{2}(m+n-3) Q_{\rho}^{m} U 
\end{align*}
where $\delta_{\rho} X=\operatorname{tr}_{g} \nabla X+m \rho^{-1}\langle X, \nabla \rho\rangle$ is the negative of the formal adjoint of the gradient with respect to $\rho^md\mu_g$,
\begin{align*}
Q_{\rho}^{m} :=-\Delta_{\rho}^m J-2\left|P\right|^{2}+\frac{m+n-1}{2}J^{2}
\end{align*}
is the weighted $Q$-curvature. If two SMMS $(\bar X,g,\rho,m)$ and $(\bar X, \hat g, \hat \rho,m)$ are pointwise conformally equivalent, that is $\hat g=e^{2\sigma}g$ and $\hat \rho=e^\sigma \rho$ for some $\sigma$, it holds
\begin{align}
\begin{split}\label{eq:conformal-eqn4}
{\widehat{L_{2, \hat\rho}^{m}}(U)=e^{-\frac{m+n+3}{2} \sigma} L_{2, \rho}^{m}\left(e^{\frac{m+n-1}{2} \sigma} U\right)} \\ {\widehat{L_{4, \hat\rho}^{m}}(U)=e^{-\frac{m+n+5}{2}\sigma} {L_{2, \rho}^{m}}\left(e^{\frac{m+n-3}{2} \sigma} U\right)}.
\end{split}
\end{align}
for all $U\in C^\infty(X)$. 

The point of working with SMMS is that there are weighted GJMS operators defined on it, which incorporate the fractional GJMS operators on $M$ as the Dirichlet-to-Neumann maps.  

Suppose $\gamma\in (0,1)$. Set $m_0=1-2\gamma$.  Denoted by $\mathcal{C}^\gamma$ be the set of all $U\in C^\infty(X)\cap C^0(\bar X)$, asymptotically near $M$,
\begin{align}\label{eq:expan-U}
 U=f+\psi \rho^{2\gamma}+o(\rho^{2\gamma})
\end{align}
for some $f,\psi\in C^\infty(M)$. We shall also use $\mathcal{C}_f^\gamma=\{U\in \mathcal{C}^{\gamma}:U|_M=f\}$. The Sobolev spaces $W^{1,2}(\bar X,\rho^{m_0}d\mu_{ g})$ are completion of $\mathcal{C}^\gamma$ with respect to the norm
\[\|U\|_{W^{1,2}}^{2} :=\int_{X}\left(|\nabla U|^{2}+U^{2}\right) \rho^{m_0} d\mu_{g}.\]
Define
\begin{align}\label{def:Q0}
    \mathcal{Q}_{\gamma}(U,V)=\int_X \left(\langle\nabla U,\nabla V\rangle+\frac{n-2\gamma}{2}J UV\right)\rho^{m_0}d\mu_{{ g } }.
\end{align}
\begin{proposition}[\citet{Case2017}]\label{prop:case-1} 
Suppose that $\gamma\in (0,1)$ and $(\bar X,g,\rho,m_0)$ is a geodesic SMMS. For any $U,V\in \mathcal{C}^\gamma$, 
\begin{align}\label{eq:Q-gamma-0}
    \int_{X} V L_{2, \rho}^{m_0} U\rho^{m_0 }d\mu_{{ g } }+\oint_{M} V (\lim_{\rho\to 0}\rho^{m_0}(-\partial_\rho U)) d\sigma_{{ h }}=\mathcal{Q}_{ \gamma}(U, V).
\end{align}
If $\lambda_1(-\Delta_{g_+})>\frac14n^2-\gamma^2$, then $\mathcal{Q}_{\gamma}(U,U)$ is bounded below in $\mathcal{C}_f^\gamma$. It holds that
\begin{align}\label{eq:gam01-case-ieq}
\kappa_\gamma\mathcal{Q}_{\gamma}(U,U)\geq \oint_M fP^{\gamma}_{{ h }}fd\sigma_{{ h }}
\end{align} 
for all $U\in W^{1,2}(\bar X,\rho^{m_0}d\mu_g)$ with $\text{Tr}\, U=f$. Equality holds if and only if $L_{2,\rho}^{m_0} U=0$.
\end{proposition}

According to \citep[Cor. 4.6]{Mayer2017}, there exists a Green's function $G_g^\gamma(x,\xi)$ of $L_{2,\rho}^{m_0}$ satisfying 
\begin{align}\label{eq:G-0}
\begin{cases}
L_{2,\rho}^{m_0}G_g^\gamma(\cdot, \xi)=0\text{ and for all }\xi\in M,\\
-\kappa_\gamma\lim\limits_{\rho\to 0} \rho^{m_0}\partial_\rho G_g^\gamma(x,\xi)=\delta_{\xi}(x).
\end{cases}
\end{align}
Here $\delta_{\xi}(x)$ is the Dirac function at $\xi$. The following estimates hold for $G_g^\gamma$,
\begin{align}\label{est:G-0}
\begin{split}
|G_g^\gamma(x,\xi)-d_g(x,\xi)^{2\gamma-n}|\leq C\max\{1, d_g(x,\xi)^{2\gamma-n+1}\},\\
|\nabla(G_g^\gamma(x,\xi)-d_g(x,\xi)^{2\gamma-n})|_g\leq Cd_g(x,\xi)^{2\gamma-n}.
\end{split}
\end{align}
{Moreover, if $\mathcal{Y}^\gamma(M,[h])>0$, then $G^\gamma_g>0$ by \citet{GonzalezNogueras2013}.}

Suppose $\gamma\in(1,2)$. Set $m_1=3-2\gamma$. Denoted by $\mathcal{C}^\gamma$ be the set of all $U\in C^\infty(X)\cap C^0(\bar X)$, asymptotically near $M$,
\begin{align}\label{eq:expan-U2}
 U=f+\psi_1 \rho^{2}+\psi_2\rho^{2\gamma}+o(\rho^{2\gamma})
\end{align}
for some $f,\psi_1,\psi_2\in C^\infty(M)$. We shall also use $\mathcal{C}_f^\gamma=\{U\in \mathcal{C}^{\gamma}:U|_M=f\}$. The Sobolev space $W^{2,2}(\bar X,\rho^{m_1}d\mu_{ g})$ is the completion of $\mathcal{C}^\gamma$ with respect to the norm
\[\|U\|_{W^{2,2}}^{2} :=\int_{X}\left(\left|\nabla^{2} U+m_1 \rho^{-1}\left(\partial_{\rho} U\right)^{2} d \rho \otimes d \rho\right|^{2}+|\nabla U|^{2}+U^{2}\right) \rho^{m_1} d\mu_g.\]
Define
\begin{align}\label{def:Q1}
    &\mathcal{Q}_{\gamma}(U,V)\notag\\
    =&\int_X\left[\left(\Delta_{\rho}^{m_1} U\right)\left(\Delta_{\rho}^{m_1} V\right)-\left(4 P-(n-2 \gamma+2) J g\right)(\nabla U, \nabla V)+\frac{n-2 \gamma}{2} Q_{\rho}^{m_1} U V\right]\rho^{m_1}d\mu_{ g}.
\end{align}

\begin{proposition}[\citet{Case2017}]\label{prop:case-2}
Suppose that $\gamma\in (1,2)$ and $(\bar X,g,\rho,m_1)$ is a geodesic SMMS. For any $U,V\in \mathcal{C}^\gamma$, 
\begin{align}\label{eq:Q-gamma-1}
    \int_{X} V L_{4, \rho}^{m_1} U\rho^{m_1}d\mu_{ g}+\oint_{M}V(\lim_{\rho\to 0} \rho^{m_1}\partial_{\rho}\Delta_{\rho}^{m_1}U)d\sigma_{{ h }}=\mathcal{Q}_{\gamma}(U, V).
\end{align}
 If $\lambda_1(-\Delta_{g_+})>\frac14n^2-(2-\gamma)^2$, then $\mathcal{Q}_{\gamma}(U,U)$ is bounded below in $\mathcal{C}_f^\gamma$. It holds that
\begin{align}
\kappa_\gamma\mathcal{Q}_{\gamma}(U,U)\geq \oint_M fP^{\gamma}_{{ h }}fd\sigma_{{ h }}
\end{align} 
for all $U\in W^{2,2}(\bar X,\rho^{m_1}d\mu_g)$ with $Tr\,U=f$. Equality holds if and only if $L_{4,\rho}^{m_1} U=0$.
\end{proposition}

Similarly, for $\gamma\in (1,2)$, one can mimic the approach in \cite[Cor. 4.6]{Mayer2017} to get a Green's function of $L_{4,\rho}^{m_1}$ satisfying
\begin{align}\label{eq:G-1}
\begin{cases}
L_{4,\rho}^{m_1}G_{g}^\gamma(\cdot,\xi)=0&\text{ in }X, \text{ for all }\xi\in M,\\
\lim\limits_{\rho\to 0} \rho^{m_1}\partial_\rho G^\gamma_{g}(\cdot,\xi)=0&\text{ on }M\backslash\{\xi\},\\
\kappa_\gamma\lim\limits_{\rho\to 0}\rho^{m_1}\Delta_{\rho}^{m_1}G^\gamma_{g}(x,\xi)=\delta_{\xi}(x)&\text{ on }M.
\end{cases}
\end{align}
The Green's function has the following estimates
\begin{align}\label{est:G-1}
\begin{split}
|G_g^\gamma(x,\xi)-d_g(x,\xi)^{2\gamma-n}|&\leq C\max\{1,d_g(x,\xi)^{2\gamma-n+1}\},\\
|\nabla(G_g^\gamma(x,\xi)-d_g(x,\xi)^{2\gamma-n})|_g&\leq Cd_g(x,\xi)^{2\gamma-n},\\
|\nabla^2(G_g^\gamma(x,\xi)-d_g(x,\xi)^{2\gamma-n})|_g&\leq Cd_g(x,\xi)^{2\gamma-n-1},\\
|\nabla^3(G_g^\gamma(x,\xi)-d_g(x,\xi)^{2\gamma-n})|_g&\leq Cd_g(x,\xi)^{2\gamma-n-2}.
\end{split}
\end{align}
Moreover, if $R_h\geq 0$ and $\mathcal{Q}^\gamma_h\geq 0$ and $\mathcal{Q}^\gamma_h\not\equiv 0$ for some $h\in [h]$, then $G_g^\gamma>0$ by \citet{CaC}.
\subsection{Energy and Bubble for Type I}\hspace*{\fill} 

Suppose that $\gamma\in(0,1)$ and $(\bar X^{n+1},g, \rho,m_0)$ is a geodesic SMMS, where $\rho$ is the geodesic defining function for a representative metric ${ h }$. Define a Yamabe energy on $\bar X$ as 
\begin{align}\label{eq:Yam-Q-2}
    \overline { \mathcal{E} } _ { { h } } ^ { \gamma } [ U ] = \frac { \kappa _ { \gamma } \mathcal{Q}_{\gamma}(U,U) } { \left( \oint _ { M } | U | ^ { \frac { 2 n } { n - 2 \gamma } } d \sigma _ { { h } } \right) ^ { \frac { n - 2 \gamma } { n } } }
\end{align}
for any $U\in W^{1,2}(X,\rho^{m_0}d\mu_g)$ such that $U\not\equiv 0$ on $M$. See the precise value of $\kappa_\gamma$ at Notation. Then $\mathcal{E} _ { { h } } ^ { \gamma } [ f ]\leq  \overline { \mathcal{E} } _ { { h } } ^ { \gamma } [ U ]$ for any $U$ has the expansion \eqref{eq:expan-U}.
Denote $N=n+1$ and $\RpN=\{x=(\bar x, x_N)| \bar x\in \mathbb{R}^n, x_N>0\}$. Recall the Sobolev trace inequality on $\RpN$  (see \cite{lieb2002sharp,cotsiolis2004best})
\begin{align}
    \left(\int_{\mathbb{R}^{n}}|U(\bar{x}, 0)|^{\frac{2 n}{n-2 \gamma}} d \bar{x}\right)^{\frac{n-2 \gamma}{n}} \leq S_{n, \gamma} \int_{0}^{\infty} \int_{\mathbb{R}^{n}} x_{N}^{1-2 \gamma}\left|\nabla U\left(\bar{x}, x_{N}\right)\right|^{2} d \bar{x} d x_{N}
\end{align}
where $S_{n,\gamma}$ denotes the optimal constant (for instance, see \citep[Cor. 5.3]{GonzalezNogueras2013}). Check our Notations for precise value.

 It is known that the above equality is attained by $U=cW_{\varepsilon,\sigma}$ for any $c\in \mathbb{R}$, $\varepsilon>0$ and $\sigma\in \mathbb{R}^n=\partial\RpN$, where $W_{\varepsilon,\sigma}$ are the \textit{bubbles} defined as
\begin{align}\label{def:bubble-exp} 
W _ { \varepsilon , \sigma } \left( \bar{x} , x _ { N } \right) & = p _ { n , \gamma } \int _ { \mathbb { R } ^ { n } } \frac { x _ { N } ^ { 2 \gamma } } { \left( | \bar{x} - \bar{y} | ^ { 2 } + x _ { N } ^ { 2 } \right) ^ { \frac { n + 2 \gamma } { 2 } } } w _ { \varepsilon , \sigma } ( \bar{y} ) d \bar{y} 
\end{align}
with 
\begin{align*}
w _ { \varepsilon , \sigma } ( \bar{x} ) : = \alpha _ { n , \gamma } \left( \frac { \varepsilon } { \varepsilon ^ { 2 } + | \bar{x} - \sigma | ^ { 2 } } \right) ^ { \frac { n - 2 \gamma } { 2 } } = W _ { \varepsilon , \sigma } ( \bar{x} , 0 ).
\end{align*}
Here $p_{n,\gamma}$ is some constant such that 
\[p_{n,\gamma}\int _ { \mathbb { R } ^ { n } } \frac { x _ { N } ^ { 2 \gamma } } { \left( | \bar{x} - \bar{y} | ^ { 2 } + x _ { N } ^ { 2 } \right) ^ { \frac { n + 2 \gamma } { 2 } } } d \bar{y} =1.\] 
We choose $\alpha_{n,\gamma}$  such that the fractional curvature of $w_{\ve,\sigma}^{\frac{4}{n-2\gamma}}|dx|^2$ is 1. The precise value $p_{n,\gamma}$ and $\alpha_{n,\gamma}$ can be found in \eqref{notation-2} in the following.
We know that $W_{\ve,\sigma}$ satisfies
\begin{align}\label{eq:K1-15}
    \left\{ \begin{array} { l l } {  \Delta_{m_0}W_{\ve,\sigma}=0 } & { \text { in } \mathbb { R } _ { + } ^ {N }, } \\  -\kappa _ { \gamma }  \lim \limits_ { x _ { N} \rightarrow 0 + } x _ { N } ^ { 1 - 2 \gamma } { \partial_N W _ { \varepsilon , \sigma } }  = ( - \Delta ) ^ { \gamma } w _ { \varepsilon , \sigma } = w _ { \varepsilon , \sigma } ^ { \frac { n + 2 \gamma } { n - 2 \gamma } }  & { \text { on } \mathbb { R } ^ { n } }. \end{array} \right.
\end{align}
Here $\Delta_{m_0}=\Delta+m_0x_N^{-1}\partial_N$ is the weighted Laplacian on $\RpN$ and $\kappa_\gamma$ is a harmless constant (see \eqref{notation-2}). For simplicity, let us denote $W_\varepsilon=W_{\varepsilon,0}$ and $w_{\varepsilon}=w_{\varepsilon,0}$. Then it is easy to see
\[w_\ve(\ve \bar x)=\ve^{-\frac{n-2\gamma}{2}}w_1( \bar x),\quad W_\ve(\ve \bar x,\ve x_N)=\ve^{-\frac{n-2\gamma}{2}}W_1(\bar x,x_N).\]
Using Lemma \ref{lem:est-bubble} in the appendix, for any nonnegative integer $k\geq 0$, one can calculate 
\begin{align}
    &\int_{B_+^N(0,\delta)}x_N^{1-2\gamma}\left[|x|^{k+2}|\nabla W_\varepsilon|^2+|x|^kW_\varepsilon^2\right] dx\notag\\
    \leq &C_{n,\gamma}\begin{cases}\ve^{k+2}& \text{if }n-2\gamma-k-2>0,\\
    \ve^{k+2}\log(\delta/\ve)&\text{if }n-2\gamma-k-2=0,\\
    \ve^{k+2} (\delta/\ve)^{2\gamma+2+k-n}&\text{if }n-2\gamma-k-2<0,\end{cases}\label{eq:mnabw}
\end{align}
for any $0<2\ve\leq  \delta<1$.

\subsection{Energy and Bubble for Type II}\hspace*{\fill} 

Suppose $\gamma\in (1,\min\{2,n/2\})$ and $(\bar X^{n+1},g, \rho,m_1)$ is a geodesic SMMS, where $\rho$ is the geodesic defining function for a representative metric ${ h }$. Define a Yamabe energy on $\bar X$ as 
\begin{align}
\begin{split}\label{eq:Yam-Q-3}
     \overline { \mathcal{E} } _ { { h } } ^ { \gamma } [ U ]= \frac{\kappa_\gamma \mathcal{Q}_{\gamma}(U,U)}{\left( \oint _ { M } | U | ^ { \frac { 2 n } { n - 2 \gamma } } d \sigma _ { { h } } \right) ^ { \frac { n - 2 \gamma } { n } }}.
\end{split}
\end{align}
for any $U\in W^{2,2}(X,\rho^{m_1}d\mu_g)$ such that $U\not\equiv 0$ on $M$. Then $\mathcal{E} _ { { h } } ^ { \gamma } [ f ]\leq  \overline { \mathcal{E} } _ { { h } } ^ { \gamma } [ U ]$ for any $U$ has the expansion \eqref{eq:expan-U2}.

%

We also have the Sobolev trace inequality for $\gamma\in (1,\min\{2,n/2\})$ see \citep{chang2011fractional, Case2017}
\begin{align}
    \left(\int_{\mathbb{R}^{n}}|U(\bar{x}, 0)|^{\frac{2 n}{n-2 \gamma}} d \bar{x}\right)^{\frac{n-2 \gamma}{n}} \leq S_{n, \gamma} \int_{\RpN} x_{N}^{3-2 \gamma}\left|\Delta_{m_1} U\left(\bar{x}, x_{N}\right)\right|^{2} d \bar{x} d x_{N}
\end{align}
where $S_{n,\gamma}$ is the optimal constant. It is also known that the equality is achieved by the \textit{bubbles} \eqref{def:bubble-exp}. In this case, however, $W_{\ve,\sigma}$ satisfies 
\begin{align}\label{eq:W2RN}
    \begin{cases}
    \Delta_{m_1}^2W_{\ve,\sigma}=0&\text{in }\RpN,\\
    W_{\ve,\sigma}=w_{\ve,\sigma}&\text{on }\Rn,\\
    \lim\limits_{x_N\to 0+} x_N^{m_1}\partial_N W_{\ve,\sigma}=0&\text{on }\Rn,\\
    (-\Delta)^\gamma w_{\ve,\sigma}=\kappa_\gamma \lim\limits_{x_N\to 0+} x_N^{m_1}\partial_N \Delta_{m_1}W_{\ve,\sigma}{=w_{\ve,\sigma}^{\frac{n+2\gamma}{n-2\gamma}}}&\text{on }\Rn.
    \end{cases}
\end{align}
 Here $\Delta_{m_1}=\Delta+m_1x_N^{-1}\partial_N$ is the weighted Laplacian on $\RpN$ and $\kappa_\gamma$ can be seen in \eqref{notation-2}. Moreover it also satisfies $\Delta_{m_0}W_{\ve,\sigma}=0$, 
which is
\begin{align}\label{eq:DW-pNW2}
    \Delta_{m_1}W_{\ve,\sigma}=2x_N^{-1}\partial_N W_{\ve,\sigma}\quad \text{in }\RpN.
\end{align}
Using Lemma \ref{lem:est-bubble}, for any integer $k\geq 0$ and $0<2\varepsilon\leq \delta<1$, one has
\begin{align}\label{eq:mnabw2}
    &\int_{B_+^N(0,\delta)}x_N^{3-2\gamma}\left[|x|^kW_\varepsilon^2+|x|^{k+2}|\nabla W_\ve|^2+|x|^{k+4}|\partial_{ij}W_\ve|^2\right] dx\notag\\
    \leq& C_{n,\gamma}\begin{cases}\ve^{k+4}& \text{if }n-2\gamma-k-4>0,\\
    \ve^{k+4}\log(\delta/\ve)&\text{if }n-2\gamma-k-4=0,\\
    \ve^{k+4} (\delta/\ve)^{2\gamma+4+k-n}&\text{if }n-2\gamma-k-4<0.\end{cases}
\end{align}
\subsection{Notations}\label{sec:notations}
The following notations are used throughout this paper
\begin{enumerate}
    \item Let $N=n+1$. For $x \in \mathbb { R } _ { + } ^ { N } : = \left\{ \left( x _ { 1 } , \ldots , x _ { n } , x _ { N } \right) \in \mathbb { R } ^ { N } : x _ { N } > 0 \right\}$ ,  we write $\bar{x} = \left( x _ { 1 } , \ldots , x _ { n } , 0 \right) \in\partial \mathbb { R } _ { + } ^ { N } \simeq \mathbb { R } ^ { n }$  and $ r = | \bar{x} |$. $i,j,k$ are indices run from 1 to $n$.
    \item $\bnpe{}$ is an open ball in $\RpN$ and $\dze{}$ is an open ball in $\Rn$.
    \item  $m_0=1-2\gamma$ and  $m_1=3-2\gamma$.
    \item Some positive constants for $0<2\gamma<n$ (see \cite{chang2011fractional})
    \begin{align} 
    d_{\gamma}&=2^{2 \gamma} \frac{\Gamma(\gamma)}{\Gamma(-\gamma)}, \quad 
    \kappa _ { \gamma } =\frac{\Gamma(\gamma-\lfloor\gamma\rfloor)}{\Gamma(\gamma+1)} \frac{(-1)^{\lfloor\gamma\rfloor+1}d_\gamma}{2^{2\lfloor\gamma\rfloor+1}(\lfloor\gamma\rfloor)!}>0. 
    \end{align}
    here $\lfloor\gamma\rfloor$ is the greatest integer less than or equal to $\gamma$. One can see that
    \[\kappa_\gamma=-\frac{d_\gamma}{2\gamma} \text{ if }\gamma\in(0,1)\quad \kappa_\gamma=\frac{d_\gamma}{8\gamma(\gamma-1)}\text{ if }\gamma\in(1,\min\{2,\frac{n}{2}\}).\]
    The following positive constant are also used for $0<2\gamma<n$
    \begin{align}
    \begin{split}\label{notation-2}
     S(n,\gamma)=&\kappa_\gamma\frac{\Gamma((n-2 \gamma) / 2)}{\Gamma((n+2 \gamma) / 2)}|vol(\mathbb{S}^n)|^{-\frac{2\gamma}{n}},\quad  p _ { n , \gamma } = \frac { \Gamma \left( \frac { n + 2 \gamma } { 2 } \right) } { \pi ^ { \frac { n } { 2 } } \Gamma ( \gamma ) }\\
    \alpha _ { n , \gamma } =&[S(n,\gamma)^{-1}\kappa_\gamma]^{\frac{n-2\gamma}{4\gamma}} \left({2^{n-1}\pi^{-\frac{n+1}{2}}\Gamma(\frac{n+1}{2})}{}\right)^{\frac{n-2\gamma}{2n}}.
    \end{split}
    \end{align}
    \item The fractional Yamabe constant for sphere
    \begin{align}\label{eq:Yamabe-sphere}
        \Ygs= \mathcal{Y} ^ { \gamma } \left( \mathbb{S} ^ { n } , \left[ g _ { c } \right] \right) = S _ { n , \gamma } ^ { - 1 } \kappa _ { \gamma } = \left( \int _ { \mathbb { R } ^ { n } } w _ { \varepsilon , \sigma } ^ { \frac { 2 n } { n - 2 \gamma } } d \bar{x} \right) ^ { \frac { 2 \gamma } { n } }.
    \end{align}
    Equivalently,
    \[\int _ { \mathbb { R } ^ { n } } w _ { \varepsilon , \sigma } ^ { \frac { 2 n } { n - 2 \gamma } } d \bar{x} = (\Ygs)^{\frac{n}{2\gamma}}.\]
    \item $\chi$ is a cut-off function has support in $\bnped$ and $\chi=1$ in $\bnpe$ and
\begin{align}\label{eq:cut-off}
    \chi_\delta=\chi\left({|x|^2}/{\delta}\right)
\end{align}
    \item Volume element on $X$ is $d\mu_g$ and on $M$ is $d\sigma_h$.
\end{enumerate}
\section{Energy estimates for the Case (I-1)}
In this section, we will derive the energy estimates for (I-1). This type of estimates will be used in Lemma \ref{lem:self-interaction} in the following.

Assume that $(\bar X^{n+1}, g,\rho, m_0)$ is a geodesic SMMS, where $\rho$ is the geodesic defining function for a representative metric $h$. 
Given any $a\in M$, there exists a Fermi coordinates $\Psi_a:\mathcal{O}(a)\to \bnped$ on some neighborhood $\mathcal{O}(a)\subset X$.   One can identify $\mathcal{O}(a)$ and $\bnped$ through $\Psi_a=(\bar x,x_N)$. It follows from \citep[Lemma 2.2 and 2.4]{Kim2018} that the following expansion of metric holds near $0$
\begin{align}
\begin{split}\label{eq:expan-metric-low}
    { g } ^{ij}(x)=&\delta_{ij}+\frac13R_{ikjl}[{ h }]x_kx_l+R_{iNjN}[{ g } ]x_N^2+O(|x|^3),\\
    \sqrt { | { g }  | } \left( x\right) =& 1  + O ( \left| x \right| ^ { 3 } ) \quad \text { in } \bnped.
    \end{split}
\end{align}
 Here $R_{ikjl}[{ h }]$ is a component of the Riemannian curvature tensor on $M$, $R_{iNjN}[{ g } ]$ is that of the Riemannian curvature tensor in $X$. Every tensor in the expansions is computed at $a=0$. Here we implicitly use the fact that $(M,{ h })\subset (\bar X,{ g } )$ is totally geodesic. Let $C_0\ve<\delta\leq \delta_0\leq 1$. Denote
 \begin{align}\label{U:I-1}
 U_{a,\ve, \delta}(x)= \chi_{\delta} W _ { \ve } ( \Psi_{a}(x))+(1-\chi_\delta(\Psi_{a}(x)))\ve^{\frac{n-2\gamma}{2}}G_g^\gamma,
 \end{align}
where $\chi_\delta$ is defined in \eqref{eq:cut-off}, $G_g^\gamma$ is the Green's function.
\begin{proposition}\label{prop:I-low}
Suppose that  $\gamma\in(0,1)$ and $n< 4+2\gamma$. For $U_{a,\ve,\delta}$ defined in \eqref{U:I-1}, if $\delta_0$ small enough and $C_0$ large enough, there exists a constant $\mathcal{C}_1>0$ such that 
\begin{align*}
   { \mathcal{E} } _ { { h } } ^ { \gamma } \left[ U_{a,\ve, \delta}\right]\leq \overline { \mathcal{E} } _ { { h } } ^ { \gamma } \left[ U_{a,\ve, \delta} \right] \leq \Ygs + \epsilon ^ { n-2\gamma } \mathcal { C }_1  ( n , \gamma,{ g } ,\delta ) + o \left( \epsilon ^ { n-2\gamma } \right).
\end{align*}
\end{proposition}
\begin{proof}
The first inequality follows from the fact that $U_{a,\ve, \delta}$ has the right expansion \eqref{eq:expan-U}. Therefore we just need to justify the second inequality. Notice the above inequality echos the fact that this is a Global case.

We adopt the notation $\mathcal{Q}(U:\Omega)$ is \eqref{def:Q0} meaning the integration over some set $\Omega\subset X$. Then  
\begin{align*}
    \mathcal{Q}_{\gamma}(U_{a,\ve,\delta})=&\mathcal{Q}_{\gamma}(W_\ve:\bnpe)+\mathcal{Q}_{\gamma}(U_{a,\ve,\delta}:\bnped\backslash \bnpe)\\
    &+\mathcal{Q}_{\gamma}(\ve^{\frac{n-2\gamma}{2}}G^\gamma_g:X\backslash\mathcal{O}_a).
\end{align*}
Using estimates in \eqref{est:G-0}, one obtains
\begin{align*}
&\mathcal{Q}_{\gamma}(\ve^{\frac{n-2\gamma}{2}}G^\gamma_a:X\backslash\mathcal{O}_a)\\
=&\ve^{n-2\gamma}\int_{X\backslash\mathcal{O}_a} \left(\left| \nabla G_a^\gamma\right|_{{ g } } ^ { 2 } +\frac{n-2\gamma}{2} J (G_a^\gamma)^2\right)\rho ^ { m_0 }d \mu_{{ g } }\leq C\ve^{n-2\gamma}\delta^{2\gamma-n}.
\end{align*}
Here $C=C(n,\gamma,{ g } )$. Similarly, by the estimates of $W_\ve$ in Lemma \ref{lem:est-bubble}, we also get
\begin{align*}
\mathcal{Q}_{\gamma}(U_{a,\ve,\delta}:\bnped\backslash \bnpe)\leq C\ve^{n-2\gamma}\delta^{2\gamma-n}.
\end{align*}
For the first term in $\mathcal{Q}_{\gamma}(U)$, applying \eqref{eq:mnabw}
\begin{align}\label{eq:E-nm-p1}
\begin{split}
    \mathcal{Q}_{\gamma}(W_\ve:\bnpe)=& \int _ { \bnpe } x _ { N } ^ { 1 - 2 \gamma }\left( \left| \nabla W _ { \ve } \right|_{{ g } } ^ { 2 } +\frac{n-2\gamma}{2}J W_\ve^2\right)d \mu_{{ g } }\\
        \leq &\int _ { \bnpe } x _ { N } ^ { 1 - 2 \gamma } \left| \nabla W _ { \ve } \right|_{{ g } } ^ { 2 } dx+C\ve^{n-2\gamma}\delta^{2\gamma-n}.
\end{split}
\end{align}
 The first term in the last inequality can be estimated by \eqref{eq:expan-metric-low} and \eqref{eq:mnabw}
\begin{align}\label{eq:E-nm-p2}
\begin{split}
   &\int _ { \bnpe } x _ { N } ^ { 1 - 2 \gamma } \left| \nabla W _ { \ve } \right|_{{ g } } ^ { 2 } d x \\
   =&\int _ { \bnpe } x _ { N } ^ { 1 - 2 \gamma } \left| \nabla W _ { \ve } \right| ^ { 2 } d x + \ve ^ { 2 }  R _ { i N j N } [ { g }  ] \int _ { B _ { + } ^ { N } \left( 0 , \delta/\ve  \right) } x _ { N } ^ { 3 - 2 \gamma } \partial _ { i } W _ { 1 } \partial _ { j } W _ { 1 } d x  + O \left( \ve ^ { 3 }(\delta/\ve)^{2\gamma+3-n} \right)\\
   \leq &\int _ { \bnpe } x _ { N } ^ { 1 - 2 \gamma } \left| \nabla W _ { \epsilon } \right| ^ { 2 } d x+C\delta^{2\gamma+2-n}\ve^{n-2\gamma},
   \end{split}
\end{align}
where $n<2\gamma+2$ is used. 
It follows from \eqref{eq:K1-15} and $x\cdot \nabla W_1\leq 0$ for $x\in \RpN{}$ that
\begin{align*}
    \int _ { \bnpe } x _ { N } ^ { 1 - 2 \gamma } \left| \nabla W _ { \epsilon } \right| ^ { 2 } d x\leq k_\gamma^{-1}\int_{\dze} w_\ve^\crexp d\bar x\leq S_{n,\gamma}^{-1}\left(\int_{\dze} w_\ve^\crexp{}d\bar x\right)^{\frac{n-2\gamma}{n}}
\end{align*}
where the last inequality follows from \eqref{eq:Yamabe-sphere}. 
On the other hand, 
\begin{align}\label{eq:E-nm-p4}
    &\oint_{M}U_{a.\ve,\delta}^\crexp d\sigma_{{ h }}\geq \int_{\dze}w_\ve^\crexp d\sigma_{{ h }}\geq \int_{\dze}w_\ve^\crexp d{\bar x}-C\ve^n\delta^{-n}.
\end{align}
Putting all estimates back to the expression of \eqref{eq:Yam-Q-2}, one could get the conclusion by taking $\ve$ small enough.
\end{proof}

\section{Energy estimates for Type II}
In this section, we will study the energy estimates for $\gamma\in (1,\min\{2,n/2\})$. Again, we need the expansion of metric.
\begin{lemma}\label{lem:metric-expan1}
Suppose $(X^{n+1},M^n, g_+)$ is a Poincar\'e-Einstein manifold with conformal infinity $(M,[{ h }])$. For a fixed point $a\in M$, there exist a representative ${ h }={ h }_a$ of the class $[{ h }]$, and the geodesic defining function $\rho_a$ near $M$ such that the metric ${ g } =\rho_a^2 g_+$ in terms of Fermi coordinates around $a$ has the following expansions
\begin{align} \label{eq:det-expan-1}
&\sqrt{|{ g } |}\left(\bar{x}, x_{N}\right)\\
=& 1-\frac{1}{2} Ric[{ g } ]_{N N ; i} x_{N}^{2} x_{i}-\frac{1}{4} Ric[{ g } ]_{N N ; i j} x_{N}^{2} x_{i}x_j\notag-\frac{1}{6} Ric[{ g } ]_{N N ; N i} x_{N}^{3} x_{i}+O\left(|x|^{5}\right)
\end{align}
and
\begin{align}
     &{ g } ^{i j}\left(\bar{x}, x_{N}\right)\notag\\
    =& \delta_{i j}+\frac{1}{3} R[{ h }]_{i k j l} x_{k} x_{l}+\frac{1}{6} R[{ h }]_{i k j l;m} x_{k} x_{l} x_{m}+R[{ g } ]_{i N j N ; k} x_{N}^{2} x_{k}\notag \\ &+\left(\frac{1}{20} R[{ h }]_{i k j l ; m q}+\frac{1}{15} R[{ h }]_{i k s l} R[{ h }]_{j m s q}\right) x_{k} x_{l} x_{m} x_{q} \notag\\
    &+\frac{1}{2} R[{ g } ]_{i N j N ; k l} x_{N}^{2} x_{k} x_{l}+ {\frac{1}{12}R[{ g } ]_{i N j N ; N N} x_{N}^{4}+O\left(|x|^5\right)}\label{eq:det-expan-2}
\end{align}
near $a$. Here all tensors are computed at $a$ and the indices $i,j,k,m,q,s$ run from 1 to $n$. Moreover, one has the following relations of the curvature
\begin{enumerate}[label=(\arabic*)]
    \item $Ric[{ h }]_{i j ; k}(a)+Ric[{ h }]_{j k ; i}(a)+Ric[{ h }]_{k i ; j}(a)=0$
    \item $\pi=0 \text { on } M, \quad\operatorname{Sym}_{i j k l}\left(Ric[{ h }]_{i j ; k l}+\frac{2}{9} R[{ h }]_{m i q j} R[{ h }]_{m k q l}\right)(a)=0$
    \item $ Ric[{ g } ]_{N N ; N}(a)=Ric[{ g } ]_{a N}(y)=Ric[{ g } ]_{NN;NN}(a)=R[{ g } ]_{;NN}(a)=0$
    \item $R[{ g } ]_{i N j N}(a)=Ric[{ g } ]_{i j}(a)=0$
    \item $R[{ g } ]_{ ; i i}(a)=-\frac{n\|W[{ h }]\|^{2}}{6(n-1)}, \quad Ric[{ g } ]_{N N ; ii}(a)=R[{ g } ]_{i N j N ; i j}(a)=-\frac{\|W[{ h }]\|^{2}}{12(n-1)}$.
\end{enumerate}
Here $||W[{ h }]||$ is the norm of the Weyl tensor of $(M,{ h })$ at $a$.
\end{lemma}

\begin{proof}
The expansion \eqref{eq:det-expan-1} and \eqref{eq:det-expan-2} are firstly found by \citet{Marques2005} in the boundary Yamabe problem. \citet{Gonzalez2018} and \citet{Kim2018} adapted them to the fractional case. Here we are just simplifying their expansion by using the fact that $(X^{n+1},M^n,g_+)$ is a P-E manifold.
\end{proof}

The expansion of Ricci tensor in Fermi coordinates
\begin{lemma}\label{lem:Ric-expan}
Suppose that $(M^n,h)\subset (\bar X^{n+1},g)$ is a totally geodesic. In the Fermi coordinates around $a\in M$, the Ricci tensor $Ric[g]_{ij}$ has the following expansion,
\begin{align*}
    Ric[{ g } ]_{ij}(\bar x,x_N)=&Ric[{ g } ]_{ij}+(Ric[{ h }]_{ij;k}+Rm[{ g } ]_{iNjN;k})x_k+Ric[{ g } ]_{ij;N}x_N\\
    &+Ric[{ g } ]_{ij;Nk}x_kx_N+\left(\frac1 2Ric[{ g } ]_{ij;NN}-2\,\text{Sym}_{ij}(Ric[g]_{jl}Rm[g]_{iNlN})\right)x_N^2\\
    &+\left(\frac12Ric[{ g } ]_{ij;kl}-\frac13\text{Sym}_{ij} (Rm[{ h } ]_{iksl}Rm[{ g } ]_{sNjN})\right)x_kx_l+O(|x|^3)
\end{align*}
where the tensor on the right hand side are all evaluated at $0$ and $1\leq i,j,k,l,s\leq n$. For the other component of $Ric[g]$, we have $Ric[g]_{iN}(\bar x,x_N)=0$ and 
\begin{align*}
    Ric[g]_{NN}(\bar x,x_N)=&Ric[g]_{NN}+Ric[g]_{NN;i}x_i+Ric[g]_{NN;N}x_N+\frac12 Ric[g]_{NN;ij}x_ix_j\\
    &+Ric[g]_{NN;Ni}x_ix_N+\frac12Ric[g]_{NN;NN}x_N^2+O(|x|^3)
\end{align*}
\end{lemma}
\begin{proof}
It follows from the Taylor expansion that 
\begin{align*}
    Ric[{ g } ]_{ij}(\bar x,x_N)
    =&Ric[{ g } ]_{ij}(\bar x,0)+\partial_N Ric[{ g } ]_{ij}(\bar x,0)x_N+\frac{1}{2}\partial^2_{NN} Ric[{ g } ]_{ij}(\bar x,0)x_N^2+O(|x|^3)
\end{align*}
For the first term, we have $Ric[{ g } ]_{ij}(\bar x,0)=Ric[{ h }]_{ij}(\bar x,0)+R[{ g } ]_{iNjN}(\bar x,0)$. Since $(\bar x,0)$ is a geodesic normal coordinates of $a$ on $M$, then $Ric[{ h }]_{ij}(0)=0$ and  \cite[Lemma 2.1]{Marques2005}, 
\begin{align*}
    Ric[{ h }]_{ij}(\bar x,0)=Ric[{ h }]_{ij;k}(0)x_k+\frac{1}{2}Ric[{ h }]_{ij;kl}(0)x_kx_l+O(|\bar x|^3).
\end{align*}
Thanks to the fact that $M$ is totally geodesic 
\begin{align*}
    Rm[{ g } ]_{iNjN}(\bar x,0)=&Rm[{ g } ]_{iNjN}(0)+Rm[{ g } ]_{iNjN;k}(0)x_k\\
    &+\left(\frac12R[{ g } ]_{iNjN;kl}-\frac13\text{Sym}_{ij} Rm[{ h }]_{iksl}Rm[{ g } ]_{sNjN}\right)x_kx_l+O(|\bar x|^3),\\
   \partial_N Ric[{ g } ]_{ij}(\bar x,0)=&Ric[g]_{ij,N}(\bar x ,0)=Ric[{ g } ]_{ij;N}(0)+Ric[{ g } ]_{ij;Nk}(0)x_k+O(|\bar x|^2).
\end{align*}
For the same reason that $M$ is totally geodesic,
\begin{align*}
 \partial_{NN}^2Ric[{ g } ]_{ij}(\bar x,0)=Ric[{ g } ]_{ij;NN}(\bar x,0)-2\text{Sym}_{ij}(Ric[g]_{jk}Rm[g]_{iNkN}(\bar x,0))+O(|\bar x|). \end{align*}
Collecting all the above expansion, one can get the expansion of $Ric[g]_{ij}$. It follows from Codazzi equation that 
\[Ric[g]_{iN}=\pi_{jj;i}-\pi_{ij;j}=0.\]
For $Ric[g]_{NN}$, one can do the expansion as $Ric[g]_{ij}$.
\end{proof}

\subsection{Case (II-1): Low dimension and Case (II-2): Locally conformally flat}\hspace*{\fill} 

Suppose $C_0\ve\leq \delta<\delta_0\leq 1$. Define
\begin{align}\label{eq:U2-def}
U_{a,\ve, \delta}(x)= \chi_{\delta} W _ { \ve } ( \Psi_{a}(x))+(1-\chi_\delta(\Psi_a(x)))\ve^{\frac{n-2\gamma}{2}}G_a^\gamma
\end{align}
where $\chi_\delta$ is defined in \eqref{eq:cut-off} and $G_a^\gamma=G_{g_a}^\gamma$ is defined in $\eqref{eq:G-1}$.
 \begin{proposition}\label{prop:low-dim-2}
 Suppose $\gamma\in(1,\min\{2,n/2\})$ and $n<4+2\gamma$.  If $\delta_0$ small enough and $C_0$ large enough, then there exist a constant $\mathcal{C}_2>0$ such that 
 \begin{align}
     \mathcal{E}_{{ h }}^\gamma[U_{a,\ve,\delta}]\leq \overline { \mathcal{E} } _ { { h } } ^ { \gamma } \left[U_{a,\ve,\delta}\right] \leq \Ygs + \epsilon ^ { n-2\gamma } \mathcal { C }_2  ( n , \gamma,{ g } ,\delta ) + o \left( \epsilon ^ { n-2\gamma } \right).
 \end{align}
 \end{proposition}
 \begin{proof} Suppose $\rho$ is the geodesic defining function for $h$, then 
 \[\lim_{\rho\to 0}\rho^{m_1}\partial_\rho U_{a,\ve,\delta}=0.\]
 Then $U_{a,\ve,\delta}$ satisfies \eqref{eq:expan-U2}. It follows from proposition \ref{prop:case-2} that $\mathcal{E}_{{ h }}^\gamma[U_{a,\ve,\delta}]\leq \overline { \mathcal{E} } _ { { h } } ^ { \gamma } \left[U_{a,\ve,\delta}\right] $. Therefore we just need to prove the second inequality.
Using the estimates of $W_\ve$ in Lemma \ref{lem:est-bubble} and $G_a^\gamma$ in \eqref{est:G-1}, one can get
 \begin{align}
 \begin{split}\label{eq:Q-decomp}
\mathcal{Q}_\gamma(U_{a,\ve,\delta})\leq &\int_{\bnpe}(\Delta_{{ \rho } }^{m_1}W_\ve)^2x_N^{m_1}dx+\frac{n-2\gamma}{2}\int_{\bnpe}Q_{\rho}^{m_1}W_\ve^2x_N^{m_1}dx+\\
&-\int_{\bnpe}4(P-(n-2\gamma+2)Jg)(\nabla W_\ve,\nabla W_\ve)x_N^{m_1}dx+C\ve^{n-2\gamma}\delta^{2\gamma-n}
\end{split}
\end{align}
similar to the argument in proposition \ref{prop:I-low}.
Noticing that 
 \begin{align*}
     \Delta_{{ \rho } }^{m_1}W_\ve=\Delta_{{ g } }W_\ve+m_1x_N^{-1}\partial_N W_\ve=\Delta_{m_1}W_\ve+(\Delta_{{ g } }-\Delta_{\mathbb{R}^{n+1}})W_\ve
 \end{align*}
 and it follows from the expansion of metric \eqref{eq:det-expan-2} that 
\begin{align}\label{eq:diff-Lap}
 (\Delta_{{ g } }-\Delta_{\mathbb{R}^{n+1}})W_\ve=O(|x|^2)|\nabla^2_{\bar x} W_\varepsilon|+O(|x|)|\nabla W_\ve|.
 \end{align}
Since the estimates in Lemma \ref{lem:est-bubble}, Lemma \ref{lem:interation-low-term}, and \eqref{eq:mnabw2} 
 \begin{align*}
     \int_{\bnpe}x_N^{m_1}(\Delta_{{ \rho } }^{m_1}W_\ve)^2dx
     \leq &\int_{\bnpe}x_N^{m_1}[(\Delta_{m_1}W_\ve)^2+C|x|^2|\nabla W_\ve|^2]dx\\
     \leq&\int_{\bnpe}x_N^{m_1}(\Delta_{m_1}W_\ve)^2dx+C \ve^{n-2\gamma}\delta^{2\gamma+4-n} 
 \end{align*}
where $n<2\gamma+4$ is used. It follows from \eqref{eq:W2RN} and integration by parts that
 \begin{align*}
     \int_{\bnpe}x_N^{m_1}(\Delta_{m_1}W_\ve)^2
     =&\int_{\dze}\lim_{x_N\to 0} x_N^{m_1}\left[\partial_N\Delta_{m_1}W_\ve)W_\ve-\Delta_{m_1}W_\ve \partial_N W_\ve\right]\\
     &-\int_{\partial^+\bnpe }x_N^{m_1}(\partial_\nu \Delta_{m_1}W_\ve) W_\ve+\int_{\partial^+\bnpe}x_N^{m_1}\Delta_{m_1}W_\ve\partial_{\nu}W_\ve
 \end{align*}
 where $\nu$ is the outer unit normal of $\partial^+\bnpe=\partial\bnpe\cap\RpN$. One can get from \eqref{def:bubble-exp} that $\partial_\nu W_\ve<0$ and $\partial_\nu\Delta_{m_1}W_\ve>0$. Then the above equality implies
 \begin{align}
    \int_{\bnpe}x_N^{m_1}(\Delta_{m_1}W_\ve)^2dx\leq& \kappa_\gamma^{-1} \int_{\dze}w_\ve^{\frac{n+2\gamma}{n-2\gamma}} d\bar x
    \leq S_{n,\gamma}^{-1} \left(\int_{\dze}w_\ve^{\frac{2n}{n-2\gamma}}d\bar x\right)^{\frac{n-2\gamma}{n}}.
    \label{eq:leading-ineqS_n}
 \end{align}
 The following fact of scalar curvature at 0 can be derived from Lemma \ref{lem:metric-expan1}
 \begin{align}\label{eq:exp-R[g]}
     R[{ g } ]=R[{ g } ]_{;i}=R[{ g }  ]_{;N}=R[{ g }]_{;NN}=0, R[{ g } ]_{;ii}=-\frac{n||W[h]||^2}{6(n-1)}, 
 \end{align}
then
\begin{align}
     \int_{\bnpe}x_N^{m_1} R[{ g } ]|\nabla W_\ve|_g^2dx\leq C\ve^{n-2\gamma}\delta^{2\gamma+6-n}.
 \end{align}
 Using the symmetry of $W_\ve$ and \eqref{eq:exp-R[g]} and $Ric[{ g } ]_{NN;N}(0)=0$, and Lemma \ref{lem:Ric-expan}  
 \begin{align*}
     \int_{\bnpe}x_N^{m_1}Ric[{ g } ](\nabla W_\ve,\nabla W_\ve)dx
     =&\int_{\bnpe}x_N^{m_1}O(|x|^2|\nabla W_\ve|^2)dx
     \leq C\ve^{n-2\gamma} \delta^{2\gamma+4-n}.
 \end{align*}
 Notice $J[{ g } ]=\frac{1}{2n}R[{ g } ]$ and $P[{{ g } }]=\frac{1}{n-1}(Ric[{ g } ]-J[{ g } ])$. We obtain
 \[\int_{\bnpe}4(P-(n-2\gamma+2)Jg)(\nabla W_\ve,\nabla W_\ve)x_N^{m_1}dx\leq C\ve^{n-2\gamma} \delta^{2\gamma+4-n}.\]
It is easy to see that 
 \begin{align*}
     \int_{\bnpe}x_N^{m_1}Q_{\rho}^{m_1}W_\ve^2dx\leq C\int_{\bnpe}x_N^{m_1}W_\ve^2dx\leq C \ve^{n-2\gamma}\delta^{2\gamma+4-n}.
 \end{align*}
 Putting everything back to \eqref{eq:Yam-Q-3} and using \eqref{eq:E-nm-p4}  obtains
 \begin{align*}
   \overline { \mathcal{E} } _ { { h } } ^ { \gamma } \left[ U_{a,\ve,\delta}\right]\leq& \kappa_\gamma S_{n,\gamma}^{-1}-C\ve^{n-2\gamma}\delta^{2\gamma-n}+o(\varepsilon^{n-2\gamma})\notag\\
   =&\Ygs + \epsilon ^ { n-2\gamma } \mathcal { C }_2  ( n , \gamma,{ g } ,\delta ) + o \left( \epsilon ^ { n-2\gamma } \right).
 \end{align*}
 \end{proof}
%
 Now suppose $(M^n,[{ h }])$ is locally conformally flat. Then pick any point $a\in M$, there exists a neighborhood of $a$ in $M$ that can be identify with a Euclidean ball $D(0,\delta)$, that is ${ h }_{ij}=\delta_{ij}$ in $D(0,\delta)$. Then in a neighborhood of $a$ in $X$, identified with $B_+^N(0,\delta)$, the metric reads (see \cite{Kim2018, Mayer2017})
 \begin{align}\label{eq:lcf-g-exp}
     { g } _{i j}\left(\bar{x}, x_{N}\right)=\delta_{i j}+O\left(x_{N}^{n}\right) \quad \text { and } \quad|{ g } |=1+O\left(x_{N}^{n}\right) \quad \text {for }\left(\bar{x}, x_{N}\right) \in B_+^N(0,\delta).
 \end{align}
 \begin{proposition}\label{prop:II-2}
 Suppose that $(M^n,[h])$ is locally conformally flat, and $\gamma\in(1,\min\{2,\frac{n}{2}\})$. If $\delta_0$ small enough and $C_0$ large enough, then there exists some $\mathcal{C}_3>0$ such that 
 \begin{align*}
     { \mathcal{E} } _ { { h } } ^ { \gamma } \left[ U_{a,\ve,\delta} \right]\leq \overline { \mathcal{E} } _ { { h } } ^ { \gamma } \left[ U_{a,\ve,\delta} \right] \leq \Ygs + \ve^ { n-2\gamma } \mathcal { C }_3  ( n , \gamma,{ g } ,\delta ) + o \left( \ve ^ { n-2\gamma } \right).
 \end{align*}
 where $U_{a,\ve,\delta}$ is defined in \eqref{eq:U2-def} for $0<C_0\ve\leq \delta\leq\delta_0\leq  1$.
 \end{proposition}
\begin{proof}
The proof is similar to the one of Proposition \ref{prop:low-dim-2}, but the calculation is much more simpler because ${ g } _{ij}$ is almost Euclidean. We just highlight some differences. For the same reason we can obtain \eqref{eq:Q-decomp}.
 However, \eqref{eq:diff-Lap} will be replaced by 
 \[(\Delta_{{ g } }-\Delta_{\mathbb{R}^{n+1}})W_\ve=O(|x|^{n})|\nabla^2 W|+O(|x|^{n-1})|\nabla W_\ve|,\]
 since \eqref{eq:lcf-g-exp}.
This implies,
 \begin{align*}
     \int_{\bnpe}x_N^{m_1}(\Delta_{{ \rho } }^{m_1}W_\ve)^2dx
     \leq &\int_{\bnpe}x_N^{m_1}[(\Delta_{m_1}W_\ve)^2+C|x|^{2n-2}|\nabla W_\ve|^2]dx\\
     \leq&\int_{\bnpe}x_N^{m_1}(\Delta_{m_1}W_\ve)^2dx+C \ve^{n-2\gamma}\delta^{-2\gamma-n}. 
 \end{align*}
 Here we have used \eqref{eq:mnabw2}. The rest of the proof will be the same.
\end{proof}
\subsection{Case (II-3): Non-locally conformally flat and $n>2\gamma+4$}\hfill

We are going to use a local test function 
\begin{align}\label{U:II-3}
U_{a,\ve,\delta}(x)=\chi_{\delta}W_{\ve}(\Psi_a(x)).
\end{align}
where $\chi_\delta$ is defined in \eqref{eq:cut-off} and $\Psi_a$ is the Fermi coordinates.
\begin{theorem}\label{thm:II-3-main} 
Suppose that $\gamma\in (1,\min\{2,\frac{n}{2}\})$ and $n>4+2\gamma$. If the Weyl tensor $W[h]$ at $a$ does not vanish, then there exist $\mathcal{C}_4>0$ such that 
\begin{align*}
    {\mathcal{E}}_{{ h }}^\gamma[U_{a,\ve,\delta}]\leq \overline{\mathcal{E}}_{{ h }}^\gamma[U_{a,\ve,\delta}]\leq \mathcal{Y}_{\mathbb{S}^n}^\gamma-\ve^4\mathcal{C}_{4}(n,\gamma,g,\delta)||W[{ h }]||^2+o(\ve^4)
\end{align*}
provided $C_0\ve\leq \delta\leq\delta_0\leq 1$ for  $\delta_0$ small enough and $C_0$ large enough.
\end{theorem}
\begin{proof} For the same reason as before, we just need to show the second inequality. Adopting the notation $\mathcal{Q}_\gamma({U:\Omega})$ in \eqref{def:Q1},
one has
\begin{align*}
    \mathcal{Q}_{\gamma}(U_{a,\ve,\delta})=&\mathcal{Q}_{\gamma}(W_\ve:\bnpe)+\mathcal{Q}_{\gamma}(U_{a,\ve,\delta}:\bnped\backslash \bnpe).
\end{align*}
To make our proof more clear, we use the following notation  $\mathcal{Q}_{\gamma}(W_\ve:\bnpe)=\mathcal{T}_1-\mathcal{T}_2+\mathcal{T}_3$ where
\begin{align*}
    \mathcal{T}_1=&\int_{\bnpe}x_N^{m_1}(\Delta_{{ \rho } }^{m_1}W_\ve)^2d\mu_{{ g }},\\
    \mathcal{T}_2=&\int_{\bnpe}(4P-(n-2\gamma+2)J[g]g)(\nabla W_\ve,\nabla W_\ve)d\mu_g,\\
    \mathcal{T}_3=&\frac{n-2\gamma}{2}\int_{\bnpe}Q_\rho^{m_1}W_\ve^2d\mu_g.
\end{align*}
Step 1: Consider $\mathcal{T}_1$.  
Noticing \eqref{eq:det-expan-1}, one gets
\begin{align*}
    &\int_{\bnpe}x_N^{m_1}(\Delta_{{ \rho } }^{m_1}W_\ve)^2d\mu_{{ g } }=\int_{\bnpe}x_N^{m_1}(\Delta_{{ \rho } }^{m_1}W_\ve)^2\sqrt{|{ g } |}d\mu_{{ g } }\notag\\
    =&\int_{\bnpe}x_N^{m_1}(\Delta_{{ \rho } }^{m_1}W_\ve)^2dx-\frac{1}{4n} Ric[{ g } ]_{NN;ii}\int_{\bnpe}x_N^{m_1+2}r^2(\Delta_{m_1}W_\ve)^2dx+o(\ve^4).
\end{align*}
Since $4+2\gamma<n$
\begin{align*}
    \int_{\bnpe}x_N^{m_1+2}r^2(\Delta_{m_1}W_\ve)^2dx
    =&4\int_{\bnpe}x_N^{m_1}r^2(\p_NW_\ve)^2dx\notag\\
    =&4\ve^4\iRN x_N^{m_1}r^2(\p_N W_1)^2dx+o(\ve^4) 
\end{align*}
Introduce the notation (see  \citet[Lemma B.6]{Kim2018})
\begin{align}\label{eq:F5F6}
\mathcal{F}_5=\int_{\RpN}x_N^{m_1}r^2|\nabla W_\ve|^2dx,\quad \mathcal{F}_6=\int_{\RpN}x_N^{m_1}r^2(\p_rW_\ve)^2dx.
\end{align}
 Thus
\begin{align}
    \mathcal{T}_1=\int_{\bnpe}x_N^{m_1}(\Delta_{{ \rho } }^{m_1}W_\ve)^2dx- \frac{\ve^4}{n}Ric[{ g } ]_{NN;ii}(\mathcal{F}_5-\mathcal{F}_6)+o(\ve^4).\label{eq:sp1-tf1}
\end{align}
To handle the first term on the RHS, straightforward computation shows
\begin{align*}
     &\int_{\bnpe}x_N^{m_1}(\Delta_{{ \rho } }^{m_1}W_\ve)^2dx=\int_{\bnpe}x_N^{m_1}(\Delta_{m_1}W_\ve+(\Delta_{{ g } }-\Delta_{\RpN}) W_\ve)^2dx\\
     \leq &\int_{\bnpe}x_N^{m_1}[(\Delta_{m_1}W_\ve)^2+2\Delta_{m_1}W_\ve(\Delta_{{ g } }-\Delta_{\RpN}) W_\ve+((\Delta_{{ g } }-\Delta_{\RpN}) W_\ve)^2]dx\\
     =&\int_{\bnpe}x_N^{m_1}(\Delta_{m_1}W_\ve)^2dx+I_1+I_2.
 \end{align*}
Applying \eqref{eq:det-expan-2}, one can notice
\begin{align*}
    (\Delta_{{ g } }-\Delta_{\RpN})W_\ve
    =&({ g } ^{ab}-\delta_{ab})\partial_{ab}^2W_\ve+\partial_a{ g } ^{ab}\partial_b W_\ve+{ g } ^{ab}\partial_a \log\sqrt{|{ g } |}\partial_b W_\ve\\
    =&\left[\frac{1}{3}R[{ h }]_{ikjl}x_kx_l\right]\partial_{ij}^2W_\ve+O(|x|^3)|\nabla^2_{\bar x} W_\ve|+O(|x|^2)|\nabla W_\ve|
\end{align*}
Notice the following fact 
\begin{align}\label{eq:D2W}
    \partial_{ij}^2W_1=\partial^2_{rr}W_1\frac{x_ix_j}{r^2}+\partial_r W_1 (\frac{\delta_{ij}}{r}-\frac{x_ix_j}{r^3}).
\end{align}
Using the symmetry of $\partial_{ij}^2W_\ve$ and the properties in Lemma \ref{lem:metric-expan1}, $R[{ h }]_{ikjl}x_kx_l\partial_{ij}^2W_\ve=0$.
Consequently 
\begin{align*}
    I_2=\int_{\bnpe}x_N^{m_1}[(\Delta_{{ g } }-\Delta_{\RpN}) W_\ve]^2dx=o(\ve^4).
\end{align*}
Now consider $I_2$. Let 
$\left({ g } ^{i j}\right)^{(4)}$ be the fourth-order terms in the expansion 
\eqref{eq:det-expan-2} of  ${ g } ^{ij}$.
\begin{align*}
    I_1=&2\int_{\bnpe}x_N^{m_1}\Delta_{m_1}W_\ve({ g } ^{ij}-\delta^{ij})\partial_{ij}^2W_\ve dx\\
    =& 4\ve^4\int_{\RpN}x_N^{m_1-1}\p_{N}W_1({ g } ^{ij})^{(4)}\partial_{ij}^2W_1 dx+o(\ve^4)\\
    =&2\ve^4R[{ g } ]_{iNjN;kl}\int_{\RpN}x_N^{m_1+1}x_kx_l\partial_N W_1\partial_{ij}^2W_1 dx\\
    &+\frac{\ve^4}{3}R[{ g } ]_{iNiN;NN}\int_{\RpN}x_N^{m_1+3}\partial_N W_1\Delta W_1dx+o(\ve^4).
\end{align*}
It follows from Bianchi identity and $R[{ g } ]_{NN;NN}(a)=0$ that $R[{ g } ]_{iNiN;NN}(a)=0$. Therefore the second term in $I_1$ is equal to 0. 
Using \eqref{eq:D2W} and \cite[Corollary 29]{Brendle2008}, one could simplify $I_1$ as
\begin{align*}
    I_1=&\frac{2\ve^4}{n(n+2)}(R[{ g } ]_{iNiN;jj}+2R[{ g } ]_{iNjN;ij})(\mathcal{A}_3-\mathcal{A}_1)+\frac{2\ve^4}{n}Ric[{ g } ]_{iNiN;jj}\mathcal{A}_1+o(\ve^4),
\end{align*}
where we have used the notation of Lemma \ref{lem:appendix-case3} in Appendix. Lemma \ref{lem:metric-expan1} implies $R[{ g } ]_{iNiN;jj}=Ric[{ g } ]_{NN;jj}$. Therefore 
\begin{align*}
    I_1=&\frac{2\ve^4}{n}Ric_{NN;ii}[{ g } ]\mathcal{A}_1+\frac{6\ve^4}{n(n+2)}Ric_{NN;ii}[{ g } ](\mathcal{A}_3-\mathcal{A}_1)+o(\ve^4).
\end{align*}
Collecting the computation of $I_1$ and $I_2$ and inserting to \eqref{eq:sp1-tf1}
\begin{align}
    \mathcal{T}_1
    =&\int_{\bnpe}x_N^{m_1}(\Delta_{m_1}W_\ve)^2dx+\frac{\ve^4Ric[{ g } ]_{NN;ii}}{n}\left[-\mathcal{F}_5+\mathcal{F}_6+2\mathcal{A}_1+\frac{6(\mathcal{A}_3-\mathcal{A}_1)}{n+2}\right]+o(\ve^4).\label{eq:sp1-tf2}
\end{align}
Step 2: Let us deal with $\mathcal{T}_2$ and $\mathcal{T}_3$ in  $\mathcal{Q}_{\gamma}(W_\ve:\bnpe)$.
Using Lemma \ref{lem:Ric-expan}, we get
\begin{align*}
    &\int_{\bnpe}x_N^{m_1} Ric[{ g } ](\nabla W_1, \nabla W_1)dx\\
    =&\int_{\bnpe}x_N^{m_1} Ric[{ g } ]_{ij}(x)\partial_iW_1\partial_j W_1dx+\int_{\bnpe}x_N^{m_1} Ric[{ g } ]_{NN}(x)\partial_NW_1\partial_N W_1dx\\
    =&\frac{\ve^4}{2n(n+2)}\left[Ric[{ h }]_{kk;ii}+2Ric[{ h }]_{ik;ik}\right]\int_{\RpN}x_N^{m_1}r^2(\partial_rW_1)^2dx\\
    &+\frac{1}{2n}Ric[{ g } ]_{NN;ii}\int_{\RpN}x_N^{m_1}r^2(\partial_NW_1)^2dx+o(\ve^4).
\end{align*}
Since Lemma \ref{lem:metric-expan1} implies $Ric[{ h }]_{kk;ii}=R[{ h }]_{;ii}=2(n-1)Ric[{ g } ]_{NN;ii}$ and contracted Bianchi identity $2Ric[{ h }]_{ik;ik}=R[h]_{;kk}$, one can simplify the above equation 
\begin{align*}
     &\int_{\bnpe}x_N^{m_1} Ric[{ g } ](\nabla W_1,\nabla  W_1)dx
    =\left[\frac{2n-1}{2n(n+2)}\mathcal{F}_6+\frac{1}{2n}(\mathcal{F}_5-\mathcal{F}_6)\right]{\ve^4}Ric[{ g } ]_{NN;ii}+o(\ve^4).
\end{align*}
We also have 
\begin{align}
    &\int_{\bnpe} x_N^{m_1}J[g] |\nabla W_\ve|_{{ g } }^2d\mu_{{ g } }=\frac{1}{2n}\int_{\bnpe} x_N^{m_1}R[{ g } ] |\nabla W_\ve|^2dx+o(\ve^4)\notag\\
    =&\frac{1}{4n}\int_{\bnpe}x_N^{m_1}(R[{ g } ]_{;ij}x_ix_j+R[{ g } ]_{;NN}x_N^2)|\nabla W_\ve|^2dx+o(\ve^4)\notag\\
    =&\frac{R[{ g } ]_{;ii}\ve^4}{4n^2}\iRN  x_N^{m_1}r^2|\nabla W_1|^2dx+o(\ve^4)\notag\\
    =&\frac{R[{ g } ]_{;ii}\ve^4}{4n^2}\mathcal{F}_5+o(\ve^4)=\frac{Ric[{ g } ]_{NN;ii}\ve^4}{2n}\mathcal{F}_5+o(\ve^4).\notag
\end{align}
Since the Schouten tensor $P=\frac{1}{n-1}(Ric-Jg)$, 
\begin{align}
    \mathcal{T}_2=&\frac{4}{n-1}\int_{\bnpe}x_N^{m_1} Ric[{ g } ](\nabla W_1,\nabla  W_1)dx\notag\\
    &\quad-\left(\frac{4}{n-1}+n-2\gamma+2\right)\int_{\bnpe} x_N^{m_1}J[g] |\nabla W_\ve|_{{ g } }^2d\mu_{{ g } }\notag\\
    =&\left[\frac{2}{n-1}\frac{n-3}{n(n+2)}\mathcal{F}_6-\frac{n-2\gamma+2}{2n}\mathcal{F}_5\right]\ve^4 Ric[g]_{NN;ii}+o(\ve^4)\label{eq:sp2-2}
\end{align}
Also 
\begin{align}
    \mathcal{T}_3=&\frac{n-2\gamma}{2}\int_{\bnpe}x_N^{m_1} Q_{\rho_1}^{m_1}W_\ve^2d\mu_{{ g } }=\frac{n-2\gamma}{2}\int_{\bnpe}x_N^{m_1}(-\Delta_{m_1}J[{ g } ](0))W_\ve^2dx+o(\ve^4)\notag\\
    =&\frac{-(n-2\gamma)\ve^4}{4n}\iRN x_N^{m_1}(R[{ g } ]_{;ii}+R[{ g } ]_{;NN}+m_1x_N^{-1}\partial_NR[g])W_1^2dx+o(\ve^4)\notag\\
    =&-\frac{(n-2\gamma)\ve^4}{4n}R[{ g } ]_{;ii}\mathcal{F}_1+o(\ve^4)=-\frac{n-2\gamma}{2}\ve^4Ric[{ g } ]_{NN;ii}\mathcal{F}_1+o(\ve^4).\label{eq:sp2-3}
\end{align}
Here $\mathcal{F}_1=\int_{\RpN} x_N^{m_1}W_1^2dx$ (see the notation in \cite[Lemma B.6]{Kim2018}). Inserting \eqref{eq:sp1-tf2}, \eqref{eq:sp2-2}, and \eqref{eq:sp2-3} together back to \eqref{def:Q1}, 
\begin{align}
\mathcal{Q}_\gamma( W_\ve:\bnpe)=
    \int_{\bnpe}x_N^{m_1}(\Delta_{m_1}W_\ve)^2dx+{\ve^4} R[{ g } ]_{NN;ii}\mathcal{C}_4+o(\ve^4)
\end{align}
where 
\begin{align}
    \mathcal{C}_4=&\frac{1}{n}\left[-\mathcal{F}_5+\mathcal{F}_6+2\mathcal{A}_1+\frac{6(\mathcal{A}_3-\mathcal{A}_1)}{n+2}\right]\notag\\
    &-\left[\frac{2}{n-1}\frac{n-3}{n(n+2)}\mathcal{F}_6-\frac{n-2\gamma+2}{2n}\mathcal{F}_5\right]-\frac{n-2\gamma}{2}\mathcal{F}_1\label{def:C4}
\end{align}
It can be check that $\mathcal{C}>0$ when $\gamma\in (1,2)$ and $n>4+2\gamma$. See Lemma \ref{lem:C4>0} in the appendix.

Step 3: It is standard to get
\begin{align*}
    \mathcal{Q}_\gamma(\chi_{\delta} W_\ve: \bnped\backslash \bnpe)=o(\ve^4).
\end{align*}
Combining Step 1-3, and \eqref{eq:leading-ineqS_n}, we obtain
\begin{align*}
    \mathcal{Q}_\gamma(U_{a,\ve,\delta})\leq S_{n,\gamma}^{-1}\left(\int_{\dze} w_\ve^{\frac{2n}{n-2\gamma}}d\bar x\right)^{\frac{n-2\gamma}{n}}+\ve^4R[g]_{NN;ii}\mathcal{C}_4+o(\ve^4).
\end{align*}
Since we always have \eqref{eq:E-nm-p4}, and $R[g]_{NN;ii}=-||W[h]||^2/[12(n-1)]$ by Lemma \ref{lem:metric-expan1} 
\begin{align*}
    \overline{\mathcal{E}}_{{ h }}^\gamma[U_{a,\ve,\delta}]\leq \mathcal{Y}_{\mathbb{S}^n}^\gamma-\ve^4\mathcal{C}_{4}(n,\gamma,g,\delta)||W[{ h }]||^2+o(\ve^4).
\end{align*}
\end{proof}
\subsection{Case (II-4): Non-locally conformally flat and $n=2\gamma+4$}\hspace*{\fill} 

In this case we will have $n=4+2\gamma$. Since $\gamma\in (1,2)$, then it means $\gamma=\frac{3}{2}$ and $n=7$. The bubble has the following explicit form \cite{Sun2016}
\begin{align}\label{eq:W-3/2}
    W_{\ve,\sigma}=\alpha_{7,\frac{3}{2}}\left[\left(\frac{\ve}{(\ve+x_N)^2+|\bar x-\sigma|^2}\right)^2+4x_N\left(\frac{\ve}{(\ve+x_N)^2+|\bar x-\sigma|^2}\right)^3\right]
\end{align}
where $\alpha_{7,\frac{3}{2}}$ is defined in \eqref{notation-2}. We also have $m_1=0$ in this case.

\begin{theorem}\label{thm:II-4-main}
Suppose that $\gamma=\frac32$ and $n=7$. If the Weyl tensor at $a$  does not vanish, define
\[U_{a,\ve,\delta}(x)=\chi_{\delta}W_{\ve}(\Psi_a(x))\]
for $0<C_0\ve\leq \delta\leq \delta_0\leq 1$. Then there exists $\mathcal{C}_5>0$ such that
\begin{align*}
    {\mathcal{E}}_{{ h }}^\gamma[U_{a,\ve,\delta}]\leq \overline{\mathcal{E}}_{{ h }}^\gamma[U_{a,\ve,\delta}]\leq \mathcal{Y}_{\mathbb{S}^n}^\gamma-\ve^4\log(\delta/\ve)\mathcal{C}_{5}(n,\gamma,g,\delta)||W[{ h }]||^2+O(\ve^4)
\end{align*}
provided $\delta_0$ small enough and $C_0$ large enough.
\end{theorem}
\begin{proof}
Using the explicit form of $W_{\ve,\sigma}$, one  can calculate as the previous section. Step 1: Consider the leading term in $\mathcal{Q}_{\gamma}(W_\ve:\bnpe)$.
\begin{align}
    &\int_{\bnpe}(\Delta_{{ \rho } }^{m_1}W_\ve)^2d\mu_{{ g } }=\int_{\bnpe}x_N^{m_1}(\Delta_{{ \rho } }^{m_1}W_\ve)^2\sqrt{|{ g } |}d\mu_{{ g } }\notag\\
    =&\int_{\bnpe}(\Delta_{{ \rho } }^{m_1}W_\ve)^2dx-\frac{\ve^4}{n} Ric[{ g } ]_{NN;ii}\int_{B_+^N(0,\delta/\ve)}x_N^{2}r^2(\p_{N}W_1)^2dx+o(\ve^4)\notag\\
    =&\int_{\bnpe}(\Delta_{{ \rho } }^{m_1}W_\ve)^2dx-\frac{\pi}{32}\alpha_{7,\frac{3}{2}}^2|\mathbb{S}^6|\ve^4\log\left(\frac{\delta}{\ve}\right)Ric[{ g } ]_{NN;ii}+O(\ve^4).
\end{align}
where we have used the formula of Lemma \ref{lem:appendix-3/2} in the Appendix. Similarly 
\begin{align*}
    \int_{\bnpe}(\Delta_{{ \rho } }^{m_1}W_\ve)^2dx=\int_{\bnpe}(\Delta_{{ m_1 } }W_\ve)^2dx+I_1+I_2.
\end{align*}
It is easy to see $I_2=o(\ve^4)$ and 
\begin{align}
    I_1=-\frac{\pi}{32}\alpha_{7,\frac{3}{2}}^2|\mathbb{S}^6|\ve^4\log\left(\frac{\delta}{\ve}\right)Ric[{ g } ]_{NN;ii}+O(\ve^4).
\end{align}
\begin{align}
    \mathcal{T}_1=\int_{\bnpe}(\Delta_{{ \rho } }^{m_1}W_\ve)^2dx-\frac{\pi}{16}\alpha_{7,\frac{3}{2}}^2|\mathbb{S}^6|\ve^4\log\left(\frac{\delta}{\ve}\right)Ric[{ g } ]_{NN;ii}+O(\ve^4).
\end{align}
For Step 2, we have
\begin{align*}
     &\int_{\bnpe}Ric[{ g } ](\nabla W_1,\nabla W_1)dx
    =\frac{7\pi}{32}\alpha_{7,\frac{3}{2}}^2|\mathbb{S}^6|\ve^4\log\left(\frac{\delta}{\ve}\right)Ric[{ g } ]_{NN;ii}+O(\ve^4).
\end{align*}
and 
\begin{align*}
    &\int_{\bnpe} J |\nabla W_\ve|_{{ g } }^2d\mu_{{ g } }=\frac{5\pi}{32}\alpha_{7,\frac{3}{2}}^2|\mathbb{S}^6|\ve^4\log\left(\frac{\delta}{\ve}\right)Ric[{ g } ]_{NN;ii}+O(\ve^4).
\end{align*}
Hence 
\begin{align}
    \mathcal{T}_2= -\frac{43}{48}\alpha_{7,\frac{3}{2}}^2|\mathbb{S}^6|\ve^4\log\left(\frac{\delta}{\ve}\right)Ric[{ g } ]_{NN;ii}+O(\ve^4).
\end{align}
It is not hard to see that 
\begin{align*}
    \mathcal{T}_3&=\frac{n-2\gamma}{2}\int_{\bnpe} Q_{\rho_1}^{m_1}W_\ve^2d\mu_{{ g } }=-\frac{10\pi}{32}\alpha_{7,\frac{3}{2}}^2|\mathbb{S}^6|\ve^4\log\left(\frac{\delta}{\ve}\right)Ric[{ g } ]_{NN;ii}+O(\ve^4).
\end{align*}
Putting $\mathcal{T}_i$ back to $\mathcal{Q}_\gamma( W_\ve:\bnpe)$, one gets
\begin{align*}
&\mathcal{Q}_\gamma( W_\ve:\bnpe)\\
=&\int_{\bnpe}(\Delta_{m_1}W_\ve)^2dx+\frac{19\pi}{48}\alpha_{7,\frac{3}{2}}^2|\mathbb{S}^6|\ve^4\log\left(\frac{\delta}{\ve}\right)Ric[{ g } ]_{NN;ii}+O(\ve^4).
\end{align*}
The rest of proof will be the same as the last part of the proof of Theorem \ref{thm:II-3-main}. We shall omit it here.
\end{proof}

\section{Interaction estimates on Bubbles}
In this section, we will state the asymptotic analysis of Palais-Smale sequence of $\mathcal{E}_h^\gamma$. The Local cases then follows from the Ekeland Variational Principle.  Next we shall derive interaction estimates of bubbles which is crucial for the Algebraic Topological argument in the next section.
\subsection{Asymptotic analysis and Local cases}\hspace*{\fill}

Suppose $(X^{n+1},M^n,g^+)$ is a P-E manifold with conformal infinity $(M,[h])$. Assume $\rho$ is the unique geodesic defining function for a representative metric $h$. Then $(\bar X^{n+1},g=\rho^2g_+,\rho,m_1)$ is a geodesic SMMS.
Given any point $a\in M$, there is a ``good" conformal Fermi coordinates by Lemma \ref{lem:metric-expan1}. More precisely, there exists a conformal metric ${ h }_a\in[{ h }]$ and $\rho_a$ the associated unique geodesic definition function such that
\[g_a=\rho_a^2g_+, \quad g_a|_M={ h }_a,\quad g_a=d\rho_a^2+h_{\rho_a} \text{ near }M\]
Since $h_a\in [{ h }]$, one may assume $h_a=\phi_a^\frac{4}{n-2\gamma}{ h }$. One can see that $g_a=(\rho_a/\rho)^2 \rho^2g_+$. Letting $\rho\to 0$, we get
\[{ h }_a=\lim_{\rho\to0}\left(\frac{\rho_a}{\rho}\right)^2{ h }\quad\text{ on }M.\]
So we may naturally extend $\phi_a=(\rho_a/\rho)^{\frac{n-2\gamma}{2}}$ onto $X$. It is known that the map $a\to \phi_a$ and $g_a$ is $C^0$. By the expansion of metric \eqref{eq:det-expan-2} near $a$, one knows $\phi_a(a)=1$. Therefore $|\rho_a/\rho-1|\leq C\delta$ near $a$.

Suppose $\Psi_{a}:\mathcal{O}(a)\to \bnped$ is the Fermi coordinates map, where $\mathcal{O}(a)$ is a open neighborhood of $a$ in $X$. Recall the definition of $U_{a,\ve,\delta}$ in \eqref{eq:U2-def}. 
Define 
\begin{align}\label{def:V_i}
    u_{a,\ve,\delta}=U_{a,\ve,\delta}|_M,\quad V_{a,\ve,\delta}=\left(\frac{\rho_{a_i}}{\rho}\right)^{\frac{n-2\gamma}{2}}U_{a,\ve,\delta},\quad v_{a,\ve,\delta}=V_{a,\ve,\delta}|_M.
\end{align}
By the works of \citet{Palatucci2015} and \citet{Fang2015}, it is not hard to see the following profile decomposition
\begin{lemma}\label{lem:profile-decomp}
Suppose $\{u_\nu\}\subset W^{\gamma,2}_+(M,{ h })$ be a Palais-Smale sequence for $\mathcal{E}_{{ h }}^\gamma$, that is $d\mathcal{E}_h^\gamma [u_\nu]\to 0$ and $\mathcal{E}_h^\gamma[u]\to c_*$ as $\nu\to \infty$. After some normalization, we may assume  
\[\oint_M u_\nu^{\frac{2n}{n-2\gamma}}d\sigma_h=c_*^{\frac{n}{2\gamma}}.\]  
Then after passing to subsequence if necessary, there exists a $u_\infty\in W^{\gamma,2}_+(M,h)$, an integer $m\geq 0$ and a sequence $(a_{j,\nu},\ve_{j,\nu})$ for $1\leq j\leq m$ with the following properties:
\begin{enumerate}[label=(\roman*)]
\item $u_\infty$ satisfies $P_h^\gamma u_\infty=u_\infty^{\frac{n-2\gamma}{n+2\gamma}}$.
\item As $\nu\to \infty$,
\[||u_\nu-u_\infty-\sum_{j=1}^m v_{a_{j,\nu},\ve_{j,\nu},\delta}||_{W^{\gamma,2}(M,h)}\to 0,\]
\[(\mathcal{E}_{{ h }}^\gamma[u_\nu])^{\frac{n}{2\gamma}}\to (\mathcal{E}_{{ h }}^\gamma[u_\infty])^{\frac{n}{2\gamma}}+m(\mathcal{Y}_{\mathbb{S}^n}^\gamma)^{\frac{n}{2\gamma}}.\]
\item For $i\neq j$
\begin{align}
    \frac{\ve_{i,\nu}}{\ve_{j,\nu}}+\frac{\ve_{j,\nu}}{\ve_{i,\nu}}+\frac{d^2_{{ h }}(a_{i,\nu},a_{j,\nu})}{\ve_{i,\nu}\ve_{j,\nu}}\to \infty,
\end{align}
where $d_h$ is the distance function on $(M,h)$.
\end{enumerate}
\end{lemma}
It follows from Ekeland Variational Principle \citep{ekeland1974variational} that 
\begin{lemma}\label{lem:ekeland}
There exists a Palais-Smale sequence at level  $\mathcal{Y}^{\gamma}(M, [{ h }])$.
\end{lemma}

After the existence of Palais-Smale sequence at level $\mathcal{Y}^{\gamma}(M, [{ h }])$, the next ingredient in this approach is the same one as in the subcritical approximations. Precisely it is the existence of a {\em variational barrier} at infinity due to the presence of local information and is the content of the following proposition.
\begin{proposition}\label{bubble}(Local information helps)\\
Under the assumption of case (II-3) and (II-4), we have there exists $a\in M$, $\ve$ and $\delta$ small enough such that
$$
\mathcal{E}^{\gamma}_{{ h }}\left(v_{a,\ve,\delta}\right)<\mathcal{Y}^{\gamma}_{\mathbb{S}^n}.
$$
\end{proposition}
\begin{proof}
It follows directly from Theorem \ref{thm:II-3-main} and \ref{thm:II-4-main}.
\end{proof}

\noindent

\begin{proof}[Proof of Local case (II-3) and (II-4)]$ $\newline
By a contradiction argument, it follows directly from Lemma \ref{lem:ekeland}, and Proposition \ref{bubble}.
\end{proof}
\begin{remark}
As in the case of the subcritical approximation technique, here also the solution obtained is a minimizer.
\end{remark}

\subsection{Estimates for Global cases}\hspace*{\fill}

For the rest of this paper, we focus on the Global cases, which are (I-1), (II-1) and (II-2).
For every $p\in \mathbb{N}^*$ and $A:=(a_1,\cdots,a_p)\in M^p=M\times\cdots\times M$, $\ve_i$, $\ve_j$, we define the following quantities
\begin{align}\label{def:veij}
    \ve_{i,j}=&\left(\frac{\ve_{i}}{\ve_{j}}+\frac{\ve_{j}}{\ve_{i}}+\frac{d^2_{{ h }}(a_{i},a_{j})}{\ve_{i}\ve_{j}}\right)^{\frac{2\gamma-n}{2}},\\
    e_{i,j}=&\kappa_\gamma\mathcal{Q}_{\gamma}(V_{a_i,\ve,\delta},V_{a_j,\ve_j,\delta}),\\
    \epsilon_{i,j}=&\oint_M (v_{a_i,\ve_i,\delta})^{\frac{n+2\gamma}{n-2\gamma}}v_{a_j,\ve_j,\delta}d\sigma_{{ h }},
\end{align}
for $i,j=1,\cdots,p$. Here and the following we always assume that $\delta$ and $\ve_0$ are fixed numbers which will be chosen later, and $\ve_i\leq \ve_0$ are small comparable to $\delta$.
\begin{lemma}\label{lem:self-interaction}(Self-action)\\
Under the assumptions of Proposition \ref{prop:I-low}, or Proposition \ref{prop:low-dim-2} or Proposition \ref{prop:II-2}, there exist $\ve_0$ small enough and $\mathcal{C}>0$ such that for any $v_{a,\ve,\delta}$ with $\ve\leq \ve_0$
\begin{enumerate}[label=(\roman*)]
    \item $\mathcal{E}_{{ h }}^\gamma(v_{a_i,\ve,\delta})\leq \mathcal{Y}_{\mathbb{S}^n}^\gamma+\mathcal{C}\delta^{2\gamma-n}\ve^{n-2\gamma}$,
    \item $\oint_M v_{a,\ve,\delta}^{\frac{2n}{n-2\gamma}}d\sigma_{{ h }}=(\Ygs)^{\frac{n}{2\gamma}}+O(\ve^n\delta^{-n})$.
\end{enumerate}
\end{lemma}
\begin{proof}
{These are just the results of the corresponding propositions.}
\end{proof}
\begin{lemma}\label{lem_higher_interaction_estimates}(Higher exponent interaction estimates)\\ There exists $\mu_0>0$ small enough such that the following estimates hold provided $\ve_{i,j}<\mu_0$ for $i\neq j$
\begin{enumerate}[label=(\roman*)]
 \item $ \oint_{M} v_{a_{i},\ve_i, \delta}^{\alpha}v_{a_{j}, \ve_j,\delta}^{\beta}d\sigma_{{ h }}=O(\ve_{i,j}^{\beta})$ 
for $\alpha +\beta=\frac{2n}{n-2\gamma}$ and $\alpha>\frac{n}{n-2\gamma}>\beta>0$,
\item $\oint_M v_{a_{i},\ve_i, \delta}^{\frac{n}{n-2\gamma}}v_{a_{j}, \ve_j,\delta}^{\frac{n}{n-2\gamma}} 
d\sigma_{{ h }}=
O(\ve^{\frac{n}{n-2\gamma}}_{i,j}\ln \ve_{i,j})
$. 
\end{enumerate}
\end{lemma}
\begin{proof}
These are just local estimates which does not involve any fractional derivative of $v$. One can borrow the proof in \cite[Lemma 5.4]{Mayer2017}. 
\end{proof}


%

\begin{lemma}\label{lem:L2gW}
Suppose that $\gamma\in(0,1)\cup (1,\min\{2,n/2\})$ and $U_{a,\ve,\delta}$ is defined in \eqref{eq:U2-def} for $C_0\ve\leq \delta\leq \delta_0$. If $\delta_0$ small enough and $C_0$ large enough, there exist $C>0$ such that the following hold
\begin{align*}
    &|L_{2,\rho_{a}}^{m_0}(U_{a,\ve,\delta})|\leq C d_{g_a}(x,a)\chi_\delta W_\ve(\Psi_a)+C\ve^{\frac{n-2\gamma}{2}}\delta^{2\gamma-n-1}\mathbf{1}_{\{\frac12\delta\leq d_{g_{a}}(x,a)\leq 4\delta\}},\\
    &\lim_{\rho_{a_i}\to 0}\rho_{a}^{m_0}\partial_{\rho_{a}}(U_{a,\ve,\delta}) =-\kappa^{-1}_\gamma\chi_\delta w_\ve
    ^{\frac{n+2\gamma}{n-2\gamma}},
\end{align*}
where $\mathbf{1}_{\Omega}$ is the characteristic function for a set $\Omega$.
\end{lemma}
\begin{proof}
Using the map $\Psi_{a}$, we can consider the problem on $B_{+}^N(0,2\delta)$ with metric $g_a$ having expansions \eqref{eq:det-expan-1} and \eqref{eq:det-expan-2}. Under this coordinates we have $\rho_a=x_N$.
It is easy to see 
\begin{align*}
    \lim_{\rho_{a}\to 0}\rho_{a}^{m_0}\partial_{\rho_{a}}(U_{a,\ve,\delta}) =-\kappa^{-1}_\gamma\chi_\delta w_\ve^{\frac{n+2\gamma}{n-2\gamma}}.
\end{align*}
For the one of $L_{2,\rho_{a}}^{m_0}(U_{a,\ve,\delta})$, similar type of estimates were derived in \cite[Prop.B.1]{brendle2005convergence}, \cite[Prop. 3.13]{almaraz2015convergence}, and \cite[Prop. 3.14]{almaraz2016convergence}. By the definition in \eqref{eq:U2-def} and \eqref{eq:G-0},  we have
\begin{align}\label{eq:L2U}
L_{2,\rho_{a}}^{m_0}(U_{a,\ve,\delta})=&\chi_\delta L_{2,\rho_{a}}^{m_0} W_\ve+2\langle\nabla \chi_\delta, \nabla(W_\ve-\ve^{\frac{n-2\gamma}{2}}G_a^\gamma)\rangle_{g_{a}} +(\Delta_{\rho_{a}}^{m_0}\chi_\delta) (W_\ve-\ve^{\frac{n-2\gamma}{2}}G_a^\gamma)\notag\\
=&I_1+I_2+I_3.
\end{align}
To handle the first term in the above equality, notice
\[ L_{2,\rho_{a}}^{m_0} W_\ve=\Delta_{\rho_{a}}^{m_0}(W_\ve)+\frac{m_0+n-1}{2}J[{g_{a}}]W_\ve. \]
We only need to calculate the above in $\bnped$. Since $W_\ve=W_\ve(|\bar x|,x_N)=W_\ve(r,x_N)$, where $r^2=x_1^2+\cdots+x_n^2$, we have (write $g_a$ as $g$ for short temporarily)
\begin{align*}
\Delta_{\rho_a}^{m_0}(W_\ve)
=&\frac{1}{\sqrt{|g|}}\partial_i(\sqrt{|g|}g^{ij}\frac{x_j}{r}\partial_r W_\ve)+\frac{1}{\sqrt{|g|}}\partial_N(\sqrt{|g|}\partial_N W_\ve)+m_0x_N^{-1}\partial_{N} W_\ve\\
=&\frac{g^{ij}x_ix_j}{r^2}\partial_{rr}^2W_\ve+\left[g^{ij}\partial_i\ln\sqrt{|g|}\frac{x_j}{r}+\partial_i\left(\frac{g^{ij}x_j}{r}\right)\right]\partial_r W_\ve+\partial_N\ln\sqrt{|g|}\partial_N W_\ve\\
&+\partial_{NN}^2W_\ve+m_0x_N^{-1}\partial_{N} W_\ve.
\end{align*}
Using $\Delta_{m_0}W_\ve=0$, the above equality leads to
\begin{align*}
\Delta_{\rho_a}^{m_0}(W_\ve)
=&\left(\frac{g^{ij}x_ix_j}{r^2}-1\right)\partial_{rr}^2W_\ve+\left[g^{ij}\partial_i\ln\sqrt{|g|}\frac{x_j}{r}+\partial_i\left(\frac{g^{ij}x_j}{r}\right)-\frac{n-1}{r}\right]\partial_r W_\ve\\
&+\partial_N\ln\sqrt{|g|}\partial_N W_\ve.
\end{align*}
Using the expansion of $g_a$ in \eqref{eq:det-expan-1} and \eqref{eq:det-expan-2}, in $\bnped$, we have
\begin{align}
\begin{split}\label{eq:DrhoW}
\Delta_{\rho_a}^{m_0}(W_\ve)=&O(|x|^3)\partial_{rr}^2W_\ve+O(|x|^2)(\partial_rW_\ve+\partial_NW_\ve)\\
=&O(|x|\ve^{\frac{n-2\gamma}{2}}(\ve^2+|x|^2)^{-\frac{n-2\gamma}{2}})=O(|x|W_\ve),
\end{split}
\end{align}
where in the second and last equality,  Lemma \ref{lem:est-bubble} is used. 
Consequently $|I_1|\leq C|x|\chi_\delta W_\ve$.

For $I_2$ and $I_3$ in \eqref{eq:L2U}, we only need to bound them in $\bnped\backslash \bnpe$. In this region, one can use \eqref{est:G-0}, \eqref{est:G-1} and \cite[Cor. 5.3]{Mayer2017}
\begin{align*}
|W_\ve-\ve^{\frac{n-2\gamma}{2}}G_a^\gamma|+|x\cdot\nabla(W_\ve-\ve^{\frac{n-2\gamma}{2}}G_a^\gamma)|\leq C\ve^{\frac{n-2\gamma}{2}}\delta^{2\gamma-n+1}.
\end{align*}
Therefore
\begin{align*}
|I_2|+|I_3|\leq C\ve^{\frac{n-2\gamma}{2}}\delta^{2\gamma-n-1}\mathbf{1}_{\{\delta\leq |x|\leq 2\delta\}},
\end{align*}
where $\mathbf{1}_{\Omega}$ is the characteristic function for a set $\Omega$. Taking $\delta<\delta_0$ small enough such that $|x|$ and $d_g(x,a)$ are comparable, one can get the conclusion.
\end{proof}

\begin{remark}\label{rem:L2gW}
Since $(\bar X, g_{a_i},\rho_{a_i},m_0)$ and $(\bar X, { g } ,\rho,m_0)$ are two geodesic SMMS which are conformal to each other, then by the conformal change property \eqref{eq:conformal-eqn4}
\begin{align*}
    L_{2,{ \rho } }^{m_0}(V_{a_i,\ve_i,\delta})=&\left({\rho_{a_i}}/{\rho}\right)^{\frac{n+4-2\gamma}{2}}L_{2,{ \rho } _{a_i}}^{m_0}(U_{a_i,\ve_i,\delta})\\
    =&O(d_g(x,a)\chi_\delta W_{\ve_i}(\Psi_{a_i}(x)))+O(\ve_i^{\frac{n-2\gamma}{2}}\delta^{2\gamma-n-1}\mathbf{1}_{\{\frac12\delta\leq d_g(x,a)\leq 4\delta\}})
    \end{align*}
    It follows from \cite[Thm 3.2]{Case2017} that $\lim_{\rho\to 0}\rho^{m_0}\partial_{\rho}$ is also conformally covariant. Then
\begin{align*}
&\lim_{\rho\to 0} \rho^{m_0}\partial_\rho V_{a_i,\ve_i,\delta}=\phi_{a_i}^{\frac{n+2\gamma}{n-2\gamma}}\lim_{\rho_{a_i}\to 0}\rho_{a_i}^{m_0}\partial_{\rho_{a_i}}(U_{a_i,\ve_i,\delta})=-\kappa^{-1}_\gamma\phi_{a_i}^{\frac{n+2\gamma}{n-2\gamma}}\chi_\delta w_{\ve_i}^{\frac{n+2\gamma}{n-2\gamma}}(\Psi_{a_i}).
\end{align*}
\end{remark}
\begin{lemma}\label{lem:interaction}(Interaction)\\
For $\gamma\in(0,1)\cup (1,\min\{2,n/2\})$, and $C_0\max\{\ve_i,\ve_j\}\leq \delta\leq \delta_0$ for some sufficiently small $\delta_0$ and large $C_0$. Assume $\ve_{i,j}\leq \mu_0$ for some small $\mu_0$
\begin{enumerate}[label=(\roman*)]
    \item $e_{i,j}=(1+O(\delta))\epsilon_{i,j}+O(\max\{\ve_i,\ve_j\}^{2\gamma}\delta^{-2\gamma}))\ve_{i,j}$,
    \item $\epsilon_{i,j}=(\Ygs)^{\frac{n}{2\gamma}}(1+O(\delta)+O(\max\{\ve_j,\ve_i\}^{2\gamma}\delta^{-2\gamma}))\ve_{i,j}$.
\end{enumerate}
\end{lemma}

\begin{proof} 
For (ii), there is no fractional derivative involved. One can use the proof from \citep[Lemma 5.5]{Mayer2017}. Now consider (i).
Let us use abbreviation $V_i=V_{a_i,\ve_i,\delta}$, $\varphi_i=\varphi_{a_i,\ve_i,\delta}$ and $W_i=W_{\ve_i}(\Psi_{a_i}(x))$.


Suppose $\gamma\in (0,1)$. It follows from \eqref{eq:Q-gamma-0} that 
\begin{align*}
    e_{i,j}=&\kappa_\gamma\int_X L_{2,{ \rho } }^{m_0}(V_i) V_j\rho^{m_0}d\mu_{{ g } }-\kappa_\gamma\oint_M\lim_{\rho\to 0}\rho^{m_0}\partial_{\rho}(V_i) V_j d\sigma_{h}. 
\end{align*}
Here by symmetry, we can assume $\ve_j\leq \ve_i$.
Since Remark \ref{rem:L2gW} and Lemma \ref{lem:interation-low-term},
\begin{align*}
\int_X L_{2,{ \rho } }^{m_0}(V_i) V_j\rho^{m_0}d\mu_{{ g } }=O(\delta)\ve_{i,j}.
\end{align*}
For the other term, one can apply Remark \ref{rem:L2gW} and Lemma \ref{lem:bdry-inter} to get
\begin{align*}
-\kappa_\gamma\oint_M\lim_{\rho\to 0}\rho^{m_0}\partial_{\rho}(V_i) V_j d\sigma_{h} =&(1+O(\delta))\oint_{M}\chi_iw_i^{\frac{n+2\gamma}{n-2\gamma}}v_jd\sigma_h\\
=&(1+O(\delta))\epsilon_{i,j}-(1+O(\delta))\oint_{M}(v_i^{\frac{n+2\gamma}{n-2\gamma}}-\chi_iw_i^{\frac{n+2\gamma}{n-2\gamma}})v_jd\sigma_h\\
=&(1+O(\delta))\epsilon_{i,j}+O(\ve_i^{2\gamma}\delta^{-2\gamma})\ve_{i,j}.
\end{align*}
Combing the above two estimates, one gets (i) when $\gamma\in(0,1)$.

Suppose $\gamma\in (1,\min\{2,n/2\})$. It follows from \eqref{eq:Q-gamma-1} that 
\begin{align}\label{eq:inter-I-II}
     e_{i,j}=&\kappa_\gamma\int_X L_{4,{ \rho } }^{m_1}(V_i) V_j\rho^{m_1}d\mu_{{ g } }+\kappa_\gamma\oint_M\lim_{\rho\to 0}\rho^{m_1}\partial_{\rho}\Delta_{\rho}^{m_1}(V_i) V_j d\sigma_{{ h }}=I_1+I_2.
 \end{align}
 \begin{claim}
 $I_1=O(\delta)\ve_{i,j}$. 
 \end{claim}
 \begin{proof}
 It follows from \citep[Thm 3.1]{CaC} that $L_{4,\rho}^{m_1}$ has the decomposition
\[L_{4,\rho}^{m_1}=L_{2,\rho}^{m_1+2}\circ L_{2,\rho}^{m_0}=L_{2,\rho}^{m_1+2}\circ L_{2,\rho}^{m_0},\]
where by definition one has
\[L_{2,\rho}^{m_1+2}=L_{2,\rho}^{m_1}-2\rho^{-1}\partial_\rho+J[g].\]
Since
\begin{align*}
L_{2,\rho_{a}}^{m_0}(U_{a,\ve,\delta})=&\chi_\delta L_{2,\rho_{a}}^{m_0} W_\ve+(1-\chi_\delta)L_{2,\rho_{a}}^{m_0}G^\gamma_{a}\\
&+2\langle\nabla \chi_\delta, \nabla(W-\ve^{\frac{n-2\gamma}{2}}G^\gamma_a)\rangle_{g_{a}} +\Delta_{\rho_{a}}^{m_0}\chi_\delta (W_\ve-\ve^{\frac{n-2\gamma}{2}}G^\gamma_a),
\end{align*}
using the estimates in \eqref{est:G-1} and Lemma \ref{lem:est-bubble}, we arrive at the following estimates in $\bnped$ which are  
\begin{align*}
&L_{4,\rho_{a}}^{m_1}(U_{a,\ve,\delta})\\
=&L_{2,\rho_{a}}^{m_1+2}(\chi_\delta L_{2,\rho_{a}}^{m_0} W_\ve+(1-\chi_\delta)L_{2,\rho_{a}}^{m_0}G^\gamma_{a})+O(\ve^{\frac{n-2\gamma}{2}}\delta^{2\gamma-n-3}\mathbf{1}_{\{\delta\leq |x|\leq 2\delta\}})\\
=&\chi_\delta L_{2,\rho_{a}}^{m_1+2}(L_{2,\rho_{a}}^{m_0} W_\ve)+2\langle \nabla \chi_\delta ,\nabla (L_{2,\rho_{a}}^{m_0} (W_\ve-G^\gamma_a))\rangle_{g_{a}}+(\Delta_{\rho_{a}}^{m_1+2}\chi_i)L_{2,\rho_{a}}^{m_0} (W_\ve-G^\gamma_a)\\
&+O(\ve^{\frac{n-2\gamma}{2}}\delta^{2\gamma-n-3}\mathbf{1}_{\{\delta\leq |x|\leq 2\delta\}})\\
=&\chi_\delta L_{2,\rho_{a}}^{m_1+2}(L_{2,\rho_{a}}^{m_0} W_\ve)+O(\ve^{\frac{n-2\gamma}{2}}\delta^{2\gamma-n-3}\mathbf{1}_{\{\delta\leq |x|\leq 2\delta\}}).
\end{align*}
By \eqref{eq:conformal-eqn4}, we have 
\begin{align*}
    L_{4,\rho}^{m_1}(V_i)=&(\rho_{a_i}/\rho)^{\frac{n+8-2\gamma}{2}}L_{4,\rho_{a_i}}^{m_1}(U_i)\\
    =&(\rho_{a_i}/\rho)^{\frac{n+8-2\gamma}{2}}[\chi_\delta L_{2,\rho_{a_i}}^{m_1+2}(L_{2,\rho_{a_i}}^{m_0} W_i)+O(\ve^{\frac{n-2\gamma}{2}}\delta^{2\gamma-n-3}\mathbf{1}_{\{\delta\leq d_{g_{a_i}}(x,a_i)\leq 2\delta\}})]\\
    =&\chi_\delta L_{2,\rho}^{m_1+2}(L_{2,\rho}^{m_0} \tilde V_i)+O(\ve^{\frac{n-2\gamma}{2}}\delta^{2\gamma-n-3}\mathbf{1}_{\{\frac12\delta\leq d_{g}(x,a_i)\leq 4\delta\}}).
\end{align*}
Here $\tilde V_i=(\rho_{a_i}/\rho)^{\frac{n-2\gamma}{2}}W_i$. Then by Lemma \ref{lem:inter-1}
\begin{align*}
    I_1=&\kappa_\gamma\int_X \chi_\delta L_{2,\rho}^{m_1+2}(L_{2,\rho}^{m_0} \tilde V_{i})V_j \rho^{m_1}d\mu_g+O(\int_X\ve_i^{\frac{n-2\gamma}{2}}\delta^{2\gamma-n-3}\mathbf{1}_{\{\frac12\delta\leq d_{g}(x,a_i)|\leq 4\delta\}}V_j\rho^{m_1}d\mu_g)\\
    = &\kappa_\gamma\int_X \chi_\delta L_{2,\rho}^{m_1+2}(L_{2,\rho}^{m_0} \tilde V_i)V_j \rho^{m_1}d\mu_g+O(\delta)\ve_{i,j}\\
    =&\kappa_\gamma\int_X \chi_\delta L_{2,\rho}^{m_1}(L_{2,\rho}^{m_0} \tilde V_i)V_j \rho^{m_1}d\mu_g-2\kappa_\gamma \int_X \chi_\delta \rho^{-1}\partial_\rho(L_{2,\rho}^{m_0} \tilde V_i)V_j\rho^{m_1}d\mu_g+O(\delta)\ve_{i,j}\\
    =&\kappa_\gamma\int_X \chi_\delta L_{2,\rho}^{m_1}(L_{2,\rho}^{m_0} \tilde V_i)V_j \rho^{m_1}d\mu_g+O(\delta)\ve_{i,j}.
\end{align*}
It follows from integration by parts that
\begin{align*}
    &\int_X \chi_\delta L_{2,\rho}^{m_1}(L_{2,\rho}^{m_0} \tilde V_i)V_j \rho^{m_1}d\mu_g-\int_X L_{2,\rho}^{m_0} (\tilde V_i) L_{2,\rho}^{m_1}(\chi_\delta V_j)\rho^{m_1}d\mu_g\\
    =&-\oint_M \lim_{\rho\to 0} \rho^{m_1}\partial_\rho (L_{2,\rho}^{m_0} \tilde V_i)V_jd\sigma_h+\oint_M \lim_{\rho\to 0} \rho^{m_1}\partial_\rho (\chi_\delta V_j)L_{2,\rho}^{m_0} \tilde V_i d\sigma_h=0.
\end{align*}
Then
\begin{align*}
    I_1=&\kappa_\gamma \int_X (L_{2,\rho}^{m_0} \tilde V_i) L_{2,\rho}^{m_1}(\chi_\delta V_j)\rho^{m_1}d\mu_g+O(\delta)\ve_{i,j}=O(\delta)\ve_{i,j}.
\end{align*} 
\end{proof}
To deal with $I_2$ in \eqref{eq:inter-I-II}, we have 
\begin{align*}
    \lim_{\rho\to 0}\rho^{m_1}\partial_{\rho}\Delta_{\rho}^{m_1}(V_i)=\phi_{a_i}^{\frac{n+2\gamma}{n-2\gamma}}\lim_{\rho_{a_i}\to 0}\rho_{a_i}^{m_1}\partial_{\rho_{a_i}}\Delta_{\rho_{{a_i}}}^{m_1}(W_i)=\kappa_\gamma^{-1}\phi_{a_i}^{\frac{n+2\gamma}{n-2\gamma}}\chi_\delta w_i^{\frac{n+2\gamma}{n-2\gamma}}.
\end{align*}
Hence
\begin{align*}
I_2=&\oint_M \phi_{a_i}^{\frac{n+2\gamma}{n-2\gamma}}\chi_\delta w_i^{\frac{n+2\gamma}{n-2\gamma}} v_j d\sigma_h=(1+O(\delta))\oint_M \chi_\delta w_i^{\frac{n+2\gamma}{n-2\gamma}} v_jd\sigma_h\\
=&(1+O(\delta))\epsilon_{i,j}+O(\ve_i^{2\gamma}\delta^{-2\gamma}))\ve_{i,j}.
\end{align*}
Inserting the estimates of $I_1$ and $I_2$ into \eqref{eq:inter-I-II}, we get the desired result.
\end{proof}

\section{Algebraic topological argument}

In this section, we will outline the algebraic topological argument by \citet{Bahri1988}. We omit some standard proofs. Readers are encouraged to find them in \cite{mayer2017barycenter}.

To introduce the {\em neighborhood of potential critical points at infinity} of $\mathcal{E}^\gamma_{{ h }}$, we first choose some $\nu_0>1$ and $\nu_0 \approx1$, and some $\mu_0>0$ and $\mu_0 \approx0$. With the later quantities fixed, for $p\in \mathbb{N}^*$, and $0<\mu\leq \mu_0$, we define $V(p, \mu)$ the $(p, \mu)$-neighborhood of potential critical points at infinity of $\mathcal{E}^\gamma_{{ h }}$ by the following formula
\begin{equation*}
\begin{split}
V(p, \mu):=\{u\in W_+^{
\gamma, 2}(M): \,&\exists\, a_1, \cdots, a_{p}\in M,\;\alpha_1, \cdots, \alpha_{p}>0,\; 0<\ve_1, \cdots,\ve_{p}\leq \mu, \\ & \Vert u-\sum_{i=1}^{p}\alpha_i v_{a_i, \ve_i,\delta}\Vert\leq \varepsilon,    \;  \frac{\alpha_i}{\alpha_j}\leq \nu_0  \; \text{ and }   \;\varepsilon_{i, j}\leq \mu,\\& i\neq j=1, \cdots, p\},
\end{split}
\end{equation*}
where $\Vert\cdot\Vert$ denotes the standard $W^{\gamma, 2}$-norm.
\vspace{8pt}

\noindent 
Next, we introduce the sublevels of our Euler-Lagrange functional corresponding to the quantized values due to the involved bubbling phenomena. They are the sets  $L_p$ ($p\in \mathbb{N}$)  defined as follows
\begin{equation*}
L_p:=\{u\in W^{\gamma, 2}_+(M):\mathcal{E}^{\gamma}_{h}[u]\leq (p+1)^{\frac{2\gamma}{n}}\mathcal{Y}^{\gamma}_{\mathbb{S}^n}\}\quad \text{ for }\quad p\geq 1,
\end{equation*}
 and 
\begin{equation*}
L_0:=\emptyset.
\end{equation*}
As in classical Calculus of Variations and classical Critical Points Theory where Ekeland Variational Principle and Deformation Lemma plays {\em dual role} in producing Palais-Smale sequences, here also for the Ekeland Variational Principle in the Calculus of Variations at Infinity underlying the Aubin-Schoen's Minimizing Technique, we have the following Deformation Lemma which plays the corresponding role in the Critical Point Theory at Infinity behind the Barycenter Technique that we are going to use. It follows from the profile decomposition (Lemma \ref{lem:profile-decomp}) and same arguments as in others applications of the algebraic topological argument of Bahri-Coron \citep{Bahri1988}.

\begin{lemma}\label{deformationlemma} (Deformation Lemma)\\
Assuming that $\mathcal{E}^\gamma_{{ h }}$ has no critical points, then for every $p\in \mathbb{N}^*$, there exists $0<\mu_p<\mu_0$  such that, for every $0<\mu\leq \mu_p$, there holds $(L_p, L_{p-1})$ retracts by deformation onto $(L_{p-1}\cup A_p, L_{p-1})$ with $V(p, \tilde \mu)\subset A_p\subset V(p, \mu)$ where $0<\tilde \mu<\frac{\mu}{4}$ is a very small positive real number and depends on $\mu$.
\end{lemma}

On the other hand, since we are in the Global case, and no variant of the Positive Mass Theorem is known to hold, then clearly there is no variational barrier available.  However, as the Mass there is an other global invariant of the variational problem which is the Interaction. Using the later information we will establish a {\bf multiple} variational  barrier estimate (see proposition \ref{eq:baryest}) which will play the {\em dual} role in the application of the Algebraic Topological argument for existence.

Now we present some topological properties of the space of formal barycenter of $M$, that we need for our barycenter technique for existence. To do that we recall that for $p\in \N^*$  the set of formal barycenters of  $M$  of order  $p$  is defined as 
 \begin{equation*}
B_{p}(M)=\left\{\sum_{i=1}^{p}\alpha_i\delta_{a_i} : a_i\in M,  \alpha_i\geq 0,   i=1,\cdots, p, \,\sum_{i=1}^{p}\alpha_i=1\right\},\quad
B_0(M)=\emptyset,
\end{equation*}
where $\delta_{a}$ for $a\in M$ is the Dirac measure at $a$.
Moreover  we have the existence of   $\mathbb{Z}_2$  orientation classes 
\begin{equation}\label{orientation_classes}
w_p\in H_{(n+1)p-1}(B_{p}(M), B_{p-1}(M))
\end{equation}
and that the cap product acts as follows 
\begin{equation}\label{eq:actioncap}
\begin{CD}
 H^l(M^p/\sigma_p))\times H_k(B_{p}(M), B_{p-1}(M))@>\frown>> H_{k-l}(B_{p}(M), B_{p-1}(M)).
 \end{CD}
\end{equation}
On the other hand, since  $M$  is a closed  $n$-dimensional manifold, we have 
\begin{equation*}
\text{an orientation class}\, 0\neq O^{*}_{M}\in H^{n}(M),
\end{equation*}
and there is a natural way to see  $O^{*}_{M}\in H^{n}(M)$  as a nontrivial element of $H^n(M^p/\sigma_p)$, see \citet[pp. 532-533]{mayer2017barycenter}, namely 
\begin{equation}\label{eq:defom1}
O^*_{M}\simeq O^*_p \quad   \text{ with } \quad 0\neq O^{*}_p\in H^{n}((M^p)/\sigma_p).
\end{equation}
Recalling \eqref{eq:actioncap}, and identifying  $O^*_{M}$  and  $O^*_p$  via  \eqref{eq:defom1}, we have the following well-known formula.
\begin{lemma}
There holds 
$$
\begin{CD}
 H^{n}((M^p)/\sigma_p)\times H_{(n+1)p-1}(B_{p}(M), B_{p-1}(M))@>\frown>> H_{(n+1)p-(n+1)}(B_{p}(M), B_{p-1}(M))\\@>\partial>>H_{(n+1)p-n-2}(B_{p-1}(M), B_{p-2}(M)),
 \end{CD}
$$
and 
\begin{equation*}
\omega_{p-1}=\partial(O^*_{M}\frown w_p).
\end{equation*}
\end{lemma}
Next  we  define for  $p\in \mathbb{N}^*$  and  $\ve >0$
\begin{equation*}
f_p(\ve) :  B_p(M)\longrightarrow W^{\gamma, 2}_+(M)
:
\sigma=\sum_{i=1}^p\alpha_i\delta_{a_i}\in B_p(M)
\longrightarrow
f_p(\ve)(\sigma)
=
\sum_{i=1}^p\alpha_i v_{a_i, \ve,\delta}.
\end{equation*}
Using the  $f_p(\ve)$, we express the multiple variational barrier in the following proposition
\begin{proposition}\label{eq:baryest}
There exists   $\nu_0>1$  such that for every  $p\in \N^*$, $p\geq 2$  and every  $0<\mu\leq \mu_0$, there exists  $\ve_p:=\ve_p(\mu)$  such that for every  $0<\ve\leq \ve_p$  and for every  $\sigma=\sum_{i=1}^p\alpha_i\delta_{a_i}\in B_p(M)$, we have
\begin{enumerate}
 \item
If  $\;\sum_{i\neq j}\varepsilon_{i, j}> \mu$  or there exist  $i_0\neq j_0$  such that  $\frac{\alpha_{i_0}}{\alpha_{j_0}}>\nu_0$, then
 $$
\mathcal{E}^\gamma_{h}\left[f_p(\ve)(\sigma)\right]\leq p^{\frac{2\gamma}{n}}\mathcal{Y}^{\gamma}_{\mathbb{S}^n}.
$$ 
 \item
If  $\;\sum_{i\neq j}\varepsilon_{i, j}\leq \mu$  and for every  $i\neq j$  we have  $\frac{\alpha_{i}}{\alpha_j}\leq\nu_0$, then
$$
\mathcal{E}^\gamma_{h}\left[f_p(\ve)(\sigma)\right]\leq p^{\frac{2\gamma}{n}}\mathcal{Y}^{\gamma}_{\mathbb{S}^n}\left(1+\mathcal{C}_6\ve^{\frac{n-2\gamma}{2}}-{\mathcal{C}_7(p-1)}\ve^{\frac{n-2\gamma}{2}}\right),
$$
where  $\mathcal{C}_6,\mathcal{C}_7>0$ depend on $n,\gamma,g,\delta$.
\end{enumerate}

\end{proposition}

\begin{proof}{Notice that in the definition of $f_p(\ve)$ we are taking all $\ve_i$ the same. }
The proof is the same as the one of Proposition 3.1 in \citep{mayer2017barycenter} using Propositions \ref{prop:case-1}, \ref{prop:case-2}  and Lemmas \ref{lem:self-interaction}, \ref{lem_higher_interaction_estimates}, \ref{lem:interaction} and Propositions \ref{prop:I-low}, \ref{prop:low-dim-2}, \ref{prop:II-2}.
\end{proof}
Now we start transporting the topology of the manifold  $M$  into the sublevels of the Euler-Lagrange functional  $\mathcal{E}_h^\gamma$  by bubbling via  $v_{a, \ve,\delta}$. But before that, we first recall the definition of the selection map defined inside the neighborhood of potential critical points at infinity. For every  $p\in \N^*$, there exists  $0<\mu_p\leq\mu_0$  such that for every  $0<\mu\leq \mu_p$ there holds
\begin{equation}\label{eq:mini}
\begin{cases}
\forall \;u\in V(p, \mu)   \;\text{the minimization problem} \quad \min_{B_{\mu}^{p}}\Vert u-\sum_{i=1}^{p}\alpha_i v_{a_i, \ve_i,\delta}\Vert \\
\text{has a solution, which is unique up to permutations,}
\end{cases}
\end{equation}
where  $B^{p}_{\mu}$  is defined as 
\begin{equation*}
\begin{split}
B_{\mu}^{p}:=\{(\bar\alpha, A, \bar \l)  \;: \; &\ve_i\leq \mu, \;i=1, \cdots, p,\\ &\frac{\alpha_i}{\alpha_j}\leq \nu_0  \; \text{ and } \; \varepsilon_{i, j}\leq \mu, \;i\neq j=1, \cdots, p\}
\end{split}
\end{equation*}
where $(\bar\alpha, A, \bar \l)\in \mathbb{R}^{p}_+\times M^p\times (0, +\infty)^{p} $ and $\nu_0$  is as in proposition \ref{eq:baryest}. Furthermore  we define the selection map  via
\begin{equation*}
s_{p}: V(p, \mu)\longrightarrow ( M)^p/\sigma_p
:
u\longrightarrow s_{p}(u)=A
 \,\text{ and } \, A  \text{ is given by } \,\eqref{eq:mini}.
\end{equation*}

Recalling \eqref{orientation_classes} we have:
\begin{lemma}\label{eq:nontrivialf1}
Assuming that  $\mathcal{E}_{ h}^{\gamma}$  has no critical points and \:$0<\mu\leq  \mu_1$, then up to taking  $\mu_1$  smaller and  $\ve_1$  smaller too, we have that for every  $0<\ve\leq \ve_1$, there holds
$$
f_1(\ve)  :  (B_1(M),  B_0(M))\longrightarrow (L_1,  L_0)
$$
is well defined and satisfies
$$
(f_1(\ve))_*(w_1)\neq 0  \quad \text{in}  \quad H_{n}(L_1,  L_0).
$$
\end{lemma}
\begin{proof}
The proof follows from the same arguments as the ones used in the proof of  Lemma 4.2  in \cite{mayer2017barycenter} by using the selection map  $s_1$, Lemma \ref{deformationlemma} and Proposition \ref{prop:I-low}, \ref{prop:low-dim-2}, \ref{prop:II-2}.
\end{proof}
Next  we use the previous lemma and {\em pile up masses}  by bubbling  via  $v_{a, \ve,\delta}$\, in a recursive way. Still recalling \eqref{orientation_classes} we have: 
\begin{lemma}\label{lem:nontrivialrecursive}
Assuming that  $\mathcal{E}^{\gamma}_{ h}$  has no critical points and  $0<\mu\leq  \mu_{p+1}$, then up to taking  $\mu_{p+1}$  smaller, and  $\ve_p$  and  $\ve_{p+1}$  smaller too, we have that for every  $0<\ve\leq \min\{\ve_p, \ve_{p+1}\}$, there holds
$$
f_{p+1}(\ve): (B_{p+1}(M),  B_{p}(M))\longrightarrow (L_{p+1},  L_{p})
$$
and 
$$
f_p(\ve): (B_p(M),  B_{p-1}(M))\longrightarrow (L_p,   L_{p-1})
$$
are well defined and satisfy
$$(f_p(\ve))_*(w_p)\neq 0  \quad \text{in} \quad    H_{np-1}(L_p,  L_{p-1})$$ implies
$$(f_{p+1}(\ve))_*(w_{p+1})\neq 0  \quad \text{in} \quad  H_{n(p+1)-1}(L_{p+1},  L_{p}).$$
\end{lemma}
\begin{proof}
The proof follows from the same arguments as the ones used in the proof of Lemma 4.3  in \cite{mayer2017barycenter}, by using the selection map  $s_p$, Lemma \ref{deformationlemma} and Proposition \ref{eq:baryest}.
\end{proof}
\vspace{6pt}

Finally we use the strength of Proposition \ref{eq:baryest} - namely point (ii) - to give a criterion ensuring that the recursive process of {{\em piling up}} masses via Lemma \ref{lem:nontrivialrecursive} will lead to a topological contradiction after a very large number of steps. 
\begin{lemma}\label{eq:largep}
Setting  
$$p^*:=[1+\frac{\mathcal{C}_6}{\mathcal{C}_7}]+1,$$ we have that  $\forall \, 0<\ve\leq\ve_{p^*}$ there holds 
$$
f_{p^*}(\ve)[B_{p^*}(M)]\subset L_{p^*-1}.
$$
\end{lemma}
\begin{proof}
The proof is a direct application of Proposition \ref{eq:baryest} .
\end{proof}
\vspace{6pt}

\noindent

\begin{proof}[Proof of Theorem \ref{gamma01} and Theorem \ref{gamma12}]$ $\newline
It follows by a contradiction argument from Lemma \ref{eq:nontrivialf1} - Lemma \ref{eq:largep}.
\end{proof}


\section{Case (I-2): low dimension in AH}
In this section, we want to show that our method could also apply to some asymptotically hyperbolic case.  Suppose $(X^{n+1},g_+)$ is an asymptotically hyperbolic manifold with conformal infinity $(M^n,[h])$. Assume also $\rho$ is the geodesic defining function of a representative metric $h$. Furthermore we require 
\begin{align}\label{eq:K1-18}
    R[g_+]+n(n+1)=o(\rho) \text{ as }\rho\to 0\text{ uniformly on }M .
\end{align}
{Then it follows from \cite[Lem. 2.3]{Kim2018} that the mean curvature $H=0$.} According to \cite[Lemma 2.2, 2.4]{Kim2018}, for any point $a\in M$, there exist $h_a\in[h]$ (write $h_a$ as $h$ for short) and the geodesic defining function $\rho_a$ near $M$ such that the metric $g=\rho_a^2g_+$ has the following expansion
\begin{align}
    { g } ^{ij}(x)=&\delta_{ij}+2\pi_{ij}x_N+\frac13R_{ikjl}[{ h }]x_kx_l+{ g } ^{ij}_{,Nk}x_Nx_k\notag\\
    &\quad+(3\pi_{ik}\pi_{kj}+R_{iNjN}[{ g } ])x_N^2+O(|x|^3)\label{eq:exp-7}\\
    \sqrt { | { g }  | } \left( x\right) =& 1  -\frac16 Ric[h]_{ij}x_ix_j-(\frac12 \| \pi \| ^ { 2 }+Ric[g]_{NN})   x _ { N } ^ { 2 }+ O (\left| x \right| ^ { 3 } )  \text { in } \bnpe.\notag
\end{align}
in terms of Fermi coordinates around $a$. Here $\pi$ is the second fundamental form of $(M,h)\subset (\bar X,g)$. Every tensor in the expansion is computed at $a=0$.

As in \eqref{U:I-1}, we define
\[U_{a,\ve,\delta}(x)=\chi_\delta W_\ve(\Psi_a(x))+\left(1-\chi_\delta(\Psi_a(x))\right)\ve^\frac{n-2\gamma}{2}G_a^\gamma(x)\]
for $C_0\ve<\delta\leq\delta_0\leq 1$.
We shall consider the case $n< 2+2\gamma$ and $\gamma\in (0,1)$, which is a Global case, notice this implies $n=3$ and $\gamma\in (\frac12,1)$.

%

%
\begin{proposition}\label{pro:E_nonum}
Suppose that $n< 2+2\gamma$ and $\gamma\in(0,1)$. If \eqref{eq:K1-18} holds and $\delta_0$ small enough and $C_0$ large enough,  then there exists a constant $\mathcal{C}_8>0$ such that 
\begin{align}
     { \mathcal{E} } _ { { h } } ^ { \gamma } \left[ U_{a,\ve,\delta}\right]\leq \overline { \mathcal{E} } _ { { h } } ^ { \gamma } \left[U_{a,\ve,\delta} \right] \leq \Ygs + \epsilon ^ { n-2\gamma } \mathcal { C }_8  ( n , \gamma,{ g } ,\delta ) + o \left( \epsilon ^ { n-2\gamma } \right).
\end{align}
\end{proposition}
\begin{proof}
The proof is similar to the one of Proposition \ref{prop:I-low}.  {The energy inequality of \eqref{eq:gam01-case-ieq} in \citep{Case2017} goes through verbatim in AH setting for $\gamma\in(0,1)$.} One just needs to use the expansion of the metric in \eqref{eq:exp-7} instead of \eqref{eq:expan-metric-low}.
\end{proof}

Once the above proposition is established, then we have the corresponding self-action estimates in Lemma \ref{lem:self-interaction}. Although \eqref{eq:DrhoW} will be changed to $O(W_\ve)$, the interaction estimates Lemma \ref{lem:interaction} still holds in this case. Therefore, one can also run the critical points at infinity approach. 
\appendix
\renewcommand{\theequation}{A-\arabic{equation}}
\setcounter{equation}{0}
\renewcommand{\thetheorem}{A-\arabic{theorem}}
\setcounter{theorem}{0}
\section{Some estimates}
In this appendix, we will provide some details for the estimates used in the previous sections. 

\begin{lemma}\label{lem:est-bubble}
Suppose $n>2\gamma$. $W_\ve=W_{\ve,0}$ is defined in \eqref{def:bubble-exp}. Denote $|x|=|\bar x|^2+x_N^2$ on $\RpN$, then 
\begin{enumerate}
    \item $W_{\ve}(\bar x,x_N)=O(\ve^{\frac{n-2\gamma}{2}}(\ve^2+|x|^2)^{-\frac{n-2\gamma}{2}})$,
    \item $\partial_N W_{\ve}(\bar x,x_N)=O(\ve^{\frac{n-2\gamma}{2}}x_N^{2\gamma-1}(\ve^2+|x|^2)^{-\frac{n}{2}})$,
    \item $\nabla_{\bar x}W_{\ve}(\bar x,x_N)=O(\ve^{\frac{n-2\gamma}{2}}(\ve^2+|x|^2)^{-\frac{n-2\gamma+1}{2}})$,
    \item $\nabla_{\bar x}^2W_{\ve}(\bar x,x_N)=O(\ve^{\frac{n-2\gamma}{2}}(\ve^2+|x|^2)^{-\frac{n-2\gamma+2}{2}})$,
    \item $\partial_N\nabla_{\bar x}^2W_{\ve}(\bar x,x_N)=O(\ve^{\frac{n-2\gamma}{2}}x_N^{2\gamma-1}(\ve^2+|x|^2)^{-\frac{n+2}{2}}) $, for $\gamma>1$.
\end{enumerate}
\end{lemma}
\begin{proof}
These estimates follow from \cite[Cor. 5.2]{Mayer2017}. One of crucial observation in \cite[(47)]{Mayer2017} is that $W_{\ve,\sigma}$ in \eqref{def:bubble-exp} can be interpreted as the interaction of standard bubbles on $\mathbb{R}^n$. 
\end{proof}
Let us use the notation $W=W_1(|\bar x|,x_N)$ and $r=|\bar x|$. 
We have the following list of formulae. Here we borrow the notations $\mathcal{F}_i$ from \cite[Lem. B.6]{Kim2018}.
\begin{lemma}\label{lem:appendix-case3}
If $n>2\gamma+4$, then
\begin{align*}
    \mathcal{A}_1=&\iRN x_N^{4-2\gamma}r\p_N W\p_rW dx=\frac{1}{4}[\frac{n}{2}\mathcal{F}_2+(\frac{n}{2}-1)\mathcal{F}_3+\mathcal{F}_7],\\
    \mathcal{A}_2=&\iRN x_N^{5-2\gamma} r\p_{NN}^2 W \p_rWdx=-(5-2\gamma)\mathcal{A}_1+\frac{n}{2}(\mathcal{F}_2-\mathcal{F}_3),\\
    \mathcal{A}_3=&\iRN x_N^{4-2\gamma}r^2\p_NW\p_{rr}^2 W dx=-(n+1)\mathcal{A}_1-\mathcal{F}_9.
\end{align*}
\end{lemma}
\begin{proof}
Integration by parts gives
\begin{align*}
    \mathcal{A}_2&=\iRN x_N^{5-2\gamma} r\p_{NN}^2 W \p_rWdx\\
    &=-(5-2\gamma)\iRN x_N^{4-2\gamma} r\p_{N} W \p_rWdx-\iRN x_N^{5-2\gamma} r\p_{N} W \p^2_{rN}Wdx\\
    &=-(5-2\gamma)\mathcal{A}_1-\frac{1}{2}\iRN x_N^{5-2\gamma}r\p_r|\p_N W|^2 dx\\
    &=-(5-2\gamma)\mathcal{A}_1+\frac{n}{2}\iRN x_N^{5-2\gamma}|\p_N W|^2 dx\\
    &=-(5-2\gamma)\mathcal{A}_1+\frac{n}{2}(\mathcal{F}_2-\mathcal{F}_3).
\end{align*}
Using \eqref{eq:DW-pNW2}, one obtains
\begin{align*}
    -(1-2\gamma)\mathcal{A}_1=&\iRN x_N^{5-2\gamma} r\Delta W\partial_r Wdx\\=&\mathcal{F}_7+(n-1)\mathcal{F}_3+\iRN x_N^{5-2\gamma} r\p_{NN}^2W\p_rWdx\\
    =&\mathcal{F}_7+(n-1)\mathcal{F}_3-(5-2\gamma)\mathcal{A}_1+\frac{n}{2}(\mathcal{F}_2-\mathcal{F}_3).
\end{align*}
One can combine the above two equalities to get $\mathcal{A}_1$ and $\mathcal{A}_2$. Similarly 
\begin{align*}
    \mathcal{A}_3=&\iRN x_N^{4-2\gamma}r^2\p_NW\p_{rr}^2 W dx\\
    =&-(n+1)\iRN  x_N^{4-2\gamma}r \p_N W\p_rWdx-\iRN x_N^{4-2\gamma} r^2\p_rW\p_{rN}^2Wdx\\
    =&-(n+1)\mathcal{A}_1-\mathcal{F}_9.
\end{align*}
\end{proof}
\begin{lemma}\label{lem:C4>0}
Suppose $n>2\gamma+4$ and $\gamma\in (1,\min\{2,n/2\})$, then $\mathcal{C}_4$ defined in \eqref{def:C4} is positive.
\end{lemma}
\begin{proof}
Inserting the expression of $\mathcal{A}_1$ and $\mathcal{A}_3$ in the previous lemma into \eqref{def:C4} gets
\begin{align*}
&n\mathcal{C}_4\\
=&-\frac{n(n-2\gamma)}{2}\mathcal{F}_1-\frac{n}{2}\mathcal{F}_2-(\frac{n}{2}-1)\mathcal{F}_3+\frac{n-2\gamma}{2}\mathcal{F}_5+\frac{n^2-n+4}{(n-1)(n+2)}\mathcal{F}_6-\mathcal{F}_7-\frac{6}{n+2}\mathcal{F}_9\\
=&I_1+I_2+\frac{n^2-n+4}{(n-1)(n+2)}\mathcal{F}_6
\end{align*}
where 
\begin{align*}
I_1=& -\frac{n}{2}\mathcal{F}_2-(\frac{n}{2}-1)\mathcal{F}_3-\mathcal{F}_7-\frac{6}{n+2}\mathcal{F}_9\\
=&-\frac{2(2-\gamma)\left(12 \gamma  (\gamma +2)+5 n^2-8 (\gamma +2) n\right)}{5(n-4)(n-4-2\gamma)(n-4+2\gamma)}A_3B_2,\\
I_2=&-\frac{n(n-2\gamma)}{2}\mathcal{F}_1+\frac{n-2\gamma}{2}\mathcal{F}_5=\frac{n(n-2 \gamma ) \left(-4 \gamma ^2+3 n^2-18 n+28\right)}{2 (\gamma +1)(n-4)(n-4-2\gamma)(n-4+2\gamma)}A_3B_2.
\end{align*}
Here we were using the expression of $\mathcal{F}_i$ in \cite[Lem. B.6]{Kim2018}. Now it is not hard to show $I_1+I_2>0$ for $n>4+2\gamma$ and $\gamma\in (1,\min\{2,n/2\})$. Consequently, $\mathcal{C}_4>0$.
\end{proof}
\begin{lemma}\label{lem:appendix-3/2}
Suppose that $0<2\ve\leq \delta\leq 1 $ and $n=2\gamma+4=7$. $W$ is defined in \eqref{eq:W-3/2}, then
\begin{align*}
\int_{ B _ { + } ^ { N } \left( 0 , \delta/\ve  \right)}W^2dx=&\frac{5\pi}{32}\alpha_{7,\frac{3}{2}}^2|\mathbb{S}^6|\log\left(\frac{\delta}{\ve}\right)+O(1),\\
\int_{ B _ { + } ^ { N } \left( 0 , \delta/\ve  \right)}r^2(\p_N W)^2dx=&\frac{7\pi}{32}\alpha_{7,\frac{3}{2}}^2|\mathbb{S}^6|\log\left(\frac{\delta}{\ve}\right)+O(1),\\
\int_{ B _ { + } ^ { N } \left( 0 , \delta/\ve  \right)}r^2(\p_r W)^2dx=&\frac{63\pi}{32}\alpha_{7,\frac{3}{2}}^2|\mathbb{S}^6|\log\left(\frac{\delta}{\ve}\right)+O(1),\\
\int_{ B _ { + } ^ { N } \left( 0 , \delta/\ve  \right)}x_Nr\p_N W\p_r Wdx=&\frac{7\pi}{32}\alpha_{7,\frac{3}{2}}^2|\mathbb{S}^6|\log\left(\frac{\delta}{\ve}\right)+O(1),\\
\int_{ B _ { + } ^ { N } \left( 0 , \delta/\ve  \right)}x_Nr^2\p_NW(\p_{rr}^2W-r^{-1}\p_r W)dx
=&-\frac{63\pi}{64}\alpha_{7,\frac{3}{2}}^2|\mathbb{S}^6|\log\left(\frac{\delta}{\ve}\right)+O(1),
\end{align*}
where $\alpha_{7,\frac{3}{2}}$ is defined in \eqref{notation-2} and $|\mathbb{S}^6|$ is the volume of 6 dimensional sphere. 
\end{lemma}
\begin{proof}
We just show how to get the second estimate, the others follow from this similarly.
\begin{align*}
    \p_r W=-4\alpha_{7,\frac{3}{2}}\frac{|\bar x|(x_N^2+8x_N+1+|\bar x|^2)}{[(1+x_N)^2+|\bar x|^2]^4},\\
    \p_N W=-4\alpha_{7,\frac{3}{2}}\frac{x_N(x_N^2+8x_N+7+|\bar x|^2)}{[(1+x_N)^2+|\bar x|^2]^4},\\
    \p_{rr}^2 W-r^{-1}\p_r W=24\alpha_{7,\frac{3}{2}}\frac{|\bar x|^2(x_N^2+10x_N+1+|\bar x|^2)}{[(1+x_N)^2+|\bar x|^2]^5}.
\end{align*}
Then
\begin{align*}
    &\int_{ B _ { + } ^ { N } \left( 0 , \delta/\ve  \right)}r^2(\p_N W)^2dx\\
    =&16\alpha_{7,\frac{3}{2}}^2\int_{\RpN\cap\{x_N\leq \delta/\ve\}}\frac{r^2x_N^2(x_N^2+8x_N+7+|\bar x|^2)^2}{[(1+x_N)^2+|\bar x|^2]^8}d\bar x dx_N+O(1)\\
    =&16\alpha_{7,\frac{3}{2}}^2\int_{0}^{\delta/\ve}\int_{\Rn}\frac{x_N^2}{(1+x_N)^3}\frac{s^2(\frac{x_N+7}{x_N+1}+s^2)^2}{(1+s^2)^{8}}s^6dsdx_N+O(1)\\
    =&16\alpha_{7,\frac{3}{2}}^2\int_{0}^{\delta/\ve}\int_{\Rn}\frac{x_N^2}{(1+x_N)^3}\frac{s^8}{(1+s^2)^{6}}dsdx_N+O(1)\\=&\frac{7\pi}{32}\alpha_{7,\frac{3}{2}}^2\log\left(\frac{\delta}{\ve}\right)+O(1).
\end{align*}
\end{proof}
Suppose $\chi_\delta$ is defined in \eqref{eq:cut-off} and $W_{\ve,\sigma}$ is defined in \eqref{def:bubble-exp}. $\Psi_a: \mathcal{O}(a)\to \bnped$ is the Fermi coordinates map. Let us use the short notation $V_i=V_{a_i,\ve_i,\delta}$ in \eqref{def:V_i}, $\chi_i=\chi_{\delta}(\Psi_{a_i})$, $W_i=W_{\ve_i}(\Psi_{a_i})$.
\begin{lemma}\label{lem:interation-low-term}
Suppose $\gamma\in(0,1)$ and $C_0\ve_j\leq C_0\ve_i\leq\delta<\delta_0$ small enough, then 
\begin{enumerate}
\item $\int_X \rho^{m_0}\chi_iW_iV_j d\mu_{{ g }}\leq C\delta^2\ve_{i,j}$,
\item $\int_X \rho^{m_0}\ve_i^{\frac{n-2\gamma}{2}}\delta^{2\gamma-n-1}\mathbf{1}_{\{\frac12\delta\leq d_{g_{a_i}}(x,a_i)\leq 4\delta\}}V_jd\mu_{{ g } }\leq C\delta\ve_{i,j}$,
\end{enumerate}
where $\ve_{i,j}$ is defined in \eqref{def:veij}.
\end{lemma}
\begin{proof}
We are using the techniques in \cite[Lem. B.4]{brendle2005convergence}. 

(1). Assume $\delta_0$ is small enough such that the support of $\chi_i$ is contained in $\{x\in X:d_g(x,a_i)\leq 4\delta\}$.
Denote
\[\mathcal{A}=\{x\in X:2d_{{ g } }(a_j,x)\leq \ve_i+d_{{ g } }(a_i,a_j)\}\cap\{d_g(x,a_i)\leq 4\delta\},\]
\[\mathcal{A}^c=\{x\in X:2d_{{ g } }(a_j,x)> \ve_i+d_{{ g } }(a_i,a_j)\}\cap\{d_g(x,a_i)\leq 4\delta\}.\]
Then it follows from Lemma \ref{lem:est-bubble} that 
\begin{align*}
&\int_X\rho^{m_0}\chi_iW_iV_jd\mu_{{ g } }=\int_{\mathcal{A}\cup\mathcal{A}^c}\rho^{m_0}\chi_iW_iV_jd\mu_{{ g } }\\
\leq&C\left(\int_{\mathcal{A}}+\int_{\mathcal{A}^c}\right)\rho^{m_0}\left(\frac{\ve_i}{\ve_i^2+d_{{ g } }(x,a_i)^2}\right)^{\frac{n-2\gamma}{2}}\left(\frac{\ve_j}{\ve_j^2+d_{{ g } }(x,a_j)^2}\right)^{\frac{n-2\gamma}{2}}d\mu_{{ g } }\\=&I_1+I_2.
\end{align*}
For $I_2$, we have
\begin{align*}
I_2\leq& C\int_{\{d_{{ g } }(x,a_i)\leq 4\delta\}}\rho^{m_0}\left(\frac{\ve_i}{\ve_i^2+d_{{ g } }(x,a_i)^2}\right)^{\frac{n-2\gamma}{2}}\left(\frac{\ve_j}{\ve_i^2+d_{{ g } }(a_i,a_j)^2}\right)^{\frac{n-2\gamma}{2}}d\mu_{{ g } }\\
\leq &C\delta^2\frac{\ve_i^{\frac{n-2\gamma}{2}}\ve_j^{\frac{n-2\gamma}{2}}}{(\ve_i^2+d_{{ g } }(a_i,a_j)^2)^{\frac{n-2\gamma}{2}}}\leq C\delta^2\ve_{i,j},
\end{align*}
where in the last inequality we used $\ve_j\leq \ve_i$. To deal with $I_1$, 
notice that on $\mathcal{A}$, one has
\[\ve_i+d_{{ g } }(x,a_i)\geq \ve_i+d_{{ g } }(a_i,a_j)-d_{{ g } }(a_j,x)\geq \frac12(\ve_i+d_{{ g } }(a_i,a_j)).\]
Consequently $d_g(a_i,a_j)\leq \delta +2d_g(x,a_i)\leq 9\delta$ and  $\mathcal{A}\subset \{d_g(x,a_j)\leq 5\delta\}$. Then  \begin{align*}
I_1\leq& C\int_{\{d_g(x,a_j)\leq 5\delta\}}\left(\frac{\ve_i}{\ve_i^2+d_{{ g } }(a_i,a_j)^2}\right)^{\frac{n-2\gamma}{2}}\left(\frac{\ve_j}{\ve_j^2+d_{{ g } }(x,a_j)^2}\right)^{\frac{n-2\gamma}{2}}d\mu_{{ g } }\\
\leq &C\delta^2\frac{\ve_i^{\frac{n-2\gamma}{2}}\ve_j^{\frac{n-2\gamma}{2}}}{(\ve_i^2+d_{{ g } }(a_i,a_j)^2)^{\frac{n-2\gamma}{2}}}\leq C\delta^2\ve_{i,j}.
\end{align*}
Combining the estimates of $I_1$ and $I_2$, we can prove (1).

(2). Taking $\delta_0$ small enough such that
\begin{align*}
&\int_X \rho^{m_0}\ve_i^{\frac{n-2\gamma}{2}}\delta^{2\gamma-n-1}\mathbf{1}_{\{\frac12\delta\leq d_{g_{a_i}}(x,a_i)\leq 4\delta\}}V_jd\mu_{{ g } }\\
\leq&C\frac{1}{\delta}\int_{ \{\frac12\delta\leq d_g(x,a_i)\leq 8\delta\}}\rho^{m_0}\left(\frac{\ve_i}{\ve_i^2+d_g(x,a_i)^2}\right)^{\frac{n-2\gamma}{2}}\left(\frac{\ve_j}{\ve_j^2+d_g(x,a_j)^2}\right)^{\frac{n-2\gamma}{2}}d\mu_g.
\end{align*}
One can use the proof of (1) without significant change to conclude (2).
\end{proof}
Similarly we have 
\begin{lemma}\label{lem:inter-1}
Suppose that $\gamma\in (1,\min\{2,n/2\})$, and $C_0\ve_j\leq C_0\ve_i\leq\delta<\delta_0$ small enough, then 
\begin{enumerate}
\item $\int_X\rho^{m_1}\chi_iW_iV_jd\mu_{g}\leq C\delta^4\ve_{i,j}$,
\item $\int_X \rho^{m_1}\ve_i^{\frac{n-2\gamma}{2}}\delta^{2\gamma-n-3}\mathbf{1}_{\{\frac12\delta\leq d_{g_{a_i}}(x,a_i)\leq 4\delta\}}V_jd\mu_{{ g } }\leq C\delta\ve_{i,j}$.
\end{enumerate}
\end{lemma}
Now let us prove some interaction estimates on the boundary.
\begin{lemma}\label{lem:bdry-inter}
Suppose that $\gamma\in(0,1)\cup(1,\min\{2,n/2\})$, and $C_0\ve_j \leq C_0\ve_i\leq\delta\leq \delta_0$, $v_i=v_{a_i,\ve_i,\delta}$ is defined in \eqref{def:V_i}. Then
\begin{align*}
    \oint_M |v_i^{\frac{n+2\gamma}{n-2\gamma}}-\chi_iw_i^{\frac{n+2\gamma}{n-2\gamma}}|v_jd\sigma_h\leq C\frac{\ve_i^{2\gamma}}{\delta^{2\gamma}}\ve_{i,j}.
\end{align*}
here $w_i$ and $v_i$ are defined in \eqref{def:bubble-exp} and  \eqref{def:V_i}.
\end{lemma}

\begin{proof} Since $M$ and $\bar X$ are smooth compact manifolds, the metric $d_g(x,a)$ and $d_h(x,a)$ for $x,a\in M$ are comparable. Notice that by Lemma \ref{lem:est-bubble},  \eqref{est:G-0} and \eqref{est:G-1}
\begin{align*}
|v_i^{\frac{n+2\gamma}{n-2\gamma}}-\chi_iw_i^{\frac{n+2\gamma}{n-2\gamma}}|\leq C(1-\chi_i)\left(\frac{\ve_i}{\ve_i^2+d_h(x,a_i)^2}\right)^{\frac{n+2\gamma}{2}}.
\end{align*}
Define
\[\mathcal{A}=\{x\in M: 2d_{{ h } }(a_j,x)\leq \ve_i+d_{{ h } }(a_i,a_j)\}\cap\{d_{{ h } }(x,a_i)\geq \delta/2\},\]
\[\mathcal{A}^c=\{x\in M: 2d_{{ h } }(a_j,x)> \ve_i+d_{{ h } }(a_i,a_j)\}\cap\{d_{{ h } }(x,a_i)\geq \delta/2\}.\]
Then
\begin{align*}
&\oint_M  (v_i^{\frac{n+2\gamma}{n-2\gamma}}-\chi_iw_i^{\frac{n+2\gamma}{n-2\gamma}})v_jd\sigma_h\\
\leq& C\int_{\mathcal{A}\cup\mathcal{A}^c}\left(\frac{\ve_i}{\ve_i^2+d_{{ h } }(x,a_i)^2}\right)^{\frac{n+2\gamma}{2}}\left(\frac{\ve_j}{\ve_j^2+d_{{ h } }(x,a_j)^2}\right)^{\frac{n-2\gamma}{2}}d\sigma_{{ h } }\\
\leq &C\int_{\mathcal{A}}\frac{\ve_i^{\frac{n+2\gamma}{2}}}{\delta^{2\gamma}(\ve_i^2+d_{{ h } }(a_i,a_j)^2)^{\frac{n}{2}}}\left(\frac{\ve_j}{\ve_j^2+d_{{ h } }(x,a_j)^2}\right)^{\frac{n-2\gamma}{2}}d\sigma_{{ h } }\\
&+C\int_{\{d_{{ h }}(x,a_i)>\delta/2\}}\left(\frac{\ve_i}{\ve_i^2+d_{{ h } }(x,a_i)^2}\right)^{\frac{n+2\gamma}{2}}\left(\frac{\ve_j}{\ve_i^2+d_{{ h } }(a_i,a_j)^2}\right)^{\frac{n-2\gamma}{2}}d\sigma_{{ h } }\\
\leq&C\frac{\ve_i^{2\gamma}}{\delta^{2\gamma}}\frac{\ve_i^{\frac{n-2\gamma}{2}}\ve_j^{\frac{n-2\gamma}{2}}}{(\ve_i^2+d_{{ h } }(a_i,a_j)^2)^{\frac{n-2\gamma}{2}}}\leq C\frac{\ve_i^{2\gamma}}{\delta^{2\gamma}}\ve_{i,j}.
\end{align*}
In the last inequality we used $\ve_j\leq \ve_i$.

\end{proof}


\small 
\bibliographystyle{plainnat}
\bibliography{references}

\end{document}